\documentclass[a4paper,10pt]{amsart}
\usepackage[matrix,arrow,tips,curve]{xy}
\usepackage[english]{babel}
\usepackage{amsmath}
\usepackage{amssymb}
\usepackage{pgf,tikz}
\usepackage{mathrsfs}
\usepackage{enumerate}
\usepackage{hyperref}
\usepackage{epstopdf}
\usepackage{multicol}
\usepackage{amssymb}
\usepackage{amsmath}
\usepackage{scalefnt}
\usepackage[T1]{fontenc}
\usepackage[latin1]{inputenc}
\usepackage{mathrsfs}
\usepackage{comment}

\input{xy}
\xyoption{all}
\newtheorem{nada}{Nada}[section]
\newtheorem{thm}[nada]{Theorem}
\newtheorem{Lemma}[nada]{Lemma}
\newtheorem{Proposition}[nada]{Proposition}
\newtheorem{Corollary}[nada]{Corollary}

\newtheorem*{thm*}{Theorem}

\theoremstyle{definition}

\newtheorem{Definition}[nada]{Definition}
\newtheorem{Remark}[nada]{Remark}

\newtheorem{Example}[nada]{Example}

\newcommand{\Ker}{\operatorname{Ker}}

\newcommand{\Sing}{\operatorname{Sing}}
\newcommand{\p}{\mathbb{P}}

\newcommand{\cA}{\mathcal{A}}
\newcommand{\cB}{\mathcal{B}}

\newcommand{\sF}{\mathscr{F}}

\newcommand{\cO}{\mathcal{O}}

\newcommand{\F}{\mathscr{F}}

\newcommand{\Z}{\mathbb{Z}}

\newenvironment{dem}{\noindent {\bf Proof:}}
{\hfill \framebox[7pt]{} \mbox{} \medskip}

\begin{document}
	\title{Split distributions on grassmann manifolds and
		smooth quadric hypersurfaces}
	\author[Alana Cavalcante]{Alana Cavalcante}
	\address{\sc Alana Cavalcante\\
		UFV\\
		Avenida Peter Henry Rolfs, s/n\\
		36570-900, Viçosa\\ Brazil}
	\email{alana@ufv.br}

	\author[Fernando Lourenço]{Fernando Lourenço}
	\address{\sc Fernando Lourenço\\
		DMM - UFLA\\
		Campus Universitário, Lavras MG, Brazil, CEP 37200-000 }
	\email{fernando.lourenco@ufla.br}

	%DMM - UFLA , Campus Universitário, Lavras MG, Brazil, CEP 37200-000
	%Email address: fernando.lourenco@ufla.br

	\date{\today}
	%\subjclass[2010]{Primary 57R30; Secondary  14J45, 57R32, 53C12}
	
	\begin{abstract}
		This work is dedicated to studying holomorphic distributions on Grassmann manifolds and smooth quadric hypersurfaces. In special, we prove, under certain conditions, when the tangent and conormal sheaves of a distribution splits as a sum of line bundles on these manifolds, generalizing the previous works on Fano threefolds and $\mathbb{P}^{n}$. We analyze how the algebro-geometric properties of the singular set of singular holomorphic distributions relate to their associated sheaves.
	\end{abstract}
	\maketitle
	
	\setcounter{tocdepth}{1}
	
	%\tableofcontents

	\section{Introduction}
	
	The key elements for the study of holomorphic distributions are their tangent sheaf, normal sheaf and singular locus. Such study have been addressed by several authors, see for instance \cite{MOJ,AMS,AMD,CJV,GP}.
	
	Concerning the relation between the tangent sheaf of a holomorphic distribution and its singular locus, L. Giraldo and A. J. Pan-Collantes showed in \cite{GP} that the tangent sheaf of a foliation of dimension $2$ on $\mathbb{P}^3$ splits if and only if its singular scheme $Z$ is aCM, i.e., has no intermediate cohomology. More recently, A. Cavalcante, M. Corrêa and S. Marchesi extended this result in \cite{AMS}, for the others smooth weighted projective complete intersection Fano threefolds.
	Using a different technique from the last references, M. Corr\^ea, M. Jardim and R. Vidal Martins also extended in \cite{CJV} the result of \cite{GP}, by considering a codimension one locally free distribution on $\mathbb{P}^n.$ 
	
	Grassmann manifolds and quadric hypersurfaces are homogeneous projective varieties with Picard number equal to one. Both are Fano varieties, i.e., their anticanonical line bundles are ample. This means that they have positive curvature in a certain sense - they are "positive" varieties from the point of view of birational geometry and the classification of algebraic varieties.
	In this work, we are interested in analyzing algebro-geometric properties of the tangent and conormal sheaves
	of  distributions on Grassmann manifolds and quadric hypersurfaces in terms of their singular schemes.
	
	Based on the strategy proof used by authors of \cite{CJV}, the Section \ref{grassm} is intended to extend the result of \cite{GP} for the Grassmannian of $k$-planes in projective space $\p^n, \; Gr(k,n).$ Using a splitting criteria, due to Ottaviani \cite{Ott2}, and the Borel-Weill-Bott Theorem we prove the converse of this theorem for Grassmannian. More precisely, we prove the following result.
	
	\begin{thm} \label{teo grass}
		Let $\sF$ be a codimension $c$ distribution on $Gr(k,n)$, such that the tangent sheaf $T_{\sF}$ is locally free and whose singular locus $Z$ has the expected codimension $c+1.$
		If $T_{\sF}$ splits as a sum of line bundles, then $Z$ is arithmetically Cohen Macaulay. Conversely, if $Z$ has pure codimension $n-k$, and under certain conditions, then $T_{\sF}$ splits as a sum of line bundles.
	\end{thm}
	
	Furthermore, O. Calvo-Andrade, M. Corrêa, and M. Jardim showed in \cite{MOJ} a cohomological criterion for the connectedness of the singular scheme of codimension one distributions on $\p^3.$ This criterion was extended for the other Fano threefolds with Picard number one in \cite{AMS}. We observed that the singular locus of a codimension one distribution the Grassmannian $Gr(2,3)$ is connected. Recovering the result of \cite{MOJ} to $\p^3,$ the precise statement is as follows.
	
	\begin{thm}
		Let $\sF$ be a distribution of codimension one on $Gr(2,3)$ and singular scheme $Z.$ If $h^2(T_\sF(-r))=0$ and the subscheme $C \subset X, \; C  \neq \emptyset,$ then $Z$ is connected and of pure dimension 1, so that $T_\sF$ is locally free.
	\end{thm}
	
		We dedicate the first part of Section \ref{quadric} to extending the result of \cite {GP} for quadric hypersurfaces of hight dimension and obtain the following result. 
	
	\begin{thm}\label{thm1}
		Let $\sF$ be a distribution of codimension one on $Q_{n}$, such that the tangent sheaf $T_{\sF}$ is locally free. If $Z =\Sing(\sF)$ is the singular 
		scheme of $\sF$, then:
		
		\begin{itemize}
			\item [(i)] If $T_{\sF}$ is a direct sum of line bundles and spinor bundle twisted by some $\mathcal{O}_{Q_{n}}(t)$, then $Z$ is arithmetically Buchsbaum, with $h^1(Q_{n}, \mathcal{I}_{Z} (r-2)) = 1$ being the only nonzero intermediate cohomology for $H^i(Q_{n},\mathcal{I}_{Z}).$
			
			\item [(ii)] If $n>3,$ and $Z$ is arithmetically Buchsbaum with $h^{1}(Q_{n}, \mathcal{I}_{Z}(r-2)) =1$ being the only nonzero intermediate cohomology for
			$H^{i}(Q_{n}, \mathcal{I}_{Z})$, and $h^{2}(Q_{n}, T_{\sF}(-2)) =  h^{n-1}(Q_{n}, T_{\sF}(-n)) = h^{n}(Q_{n},T_{\F}(-r-n))= 0$, then $T_{\sF}$ has no intermediate cohomology, i.e., $T_{\sF} = \oplus \mathcal{O}_{Q_{n}}(t_{i})\oplus S(t)$, where $S$ is a spinor bundle on $Q_{n}.$
			
		\end{itemize}

	\end{thm}
	
	Note that \cite[Theorem 3.2]{AMS} corresponds to the case $n =3$ for item $(i).$ In addition, we show when the tangent sheaf of a distribution is only split, obtaining a characterization of split distributions in terms of the parities of the quadric.
	
	\begin{Proposition}
		With the same statements of Theorem \ref{thm1} item $(ii)$, we have the following: If $n$ is even, and $H^{n-1}(Q_n, T_{\sF} \otimes S(t)) = 0$ for $t \leq -(n-1),$ then $T_{\sF}$ splits.
	\end{Proposition}
	
	Moreover, we analyze the relationship between the tangent sheaf of distributions of dimension $2$, and $3$ on certain quadric hypersurfaces, and their singular scheme $Z.$ 
	
	\begin{thm} Let $\sF$ be a holomorphic distribution on $Q_{n}$ whose its singular locus $Z$ has pure expected dimension. Then, the following hold:
		
		\begin{itemize}
			\item [(i)] If $\sF$ has dimension two, $n$ is even and $n \geq 4$ and $T_{\sF}$ splits, then $Z$ is aCM.
			\item [(ii)] If $\sF$ has dimension three, $n=5$ and $T_{\sF}$ splits, then  $h^{2}(Q_{5}, \mathcal{I}_{Z}(1-a_{1}-a_{2}-a_{3}))=1$  being the only nonzero
			intermediate cohomology for $H^{i}(Q_{5},\mathcal{I}_{Z}).$
		\end{itemize}
	\end{thm}
	
	In the second part of Section \ref{quadric}, our attention will be directed to the conormal sheaf $N^{\vee}_{\sF}$ of a one-dimensional holomorphic distribution on $Q_n.$ A. Cavalcante, M. Corrêa, and S. Marchesi showed in \cite{AMS} the relations between algebro-geometric properties of the conormal sheaf of holomorphic distributions on a smooth weighted projective complete intersection Fano threefold and their singular set. We obtain the following by extending the result to $Q_n.$
	
	\begin{thm}\label{conormal thm}
		Let $\sF$ be a distribution of dimension one on $Q_{n}.$ If $Z =\Sing(\sF)$ is the singular scheme of $\sF,$ then:
		\begin{itemize}
			\item [(i)]  If the conormal sheaf $N_{\sF}^{\vee}$ is aCM, then $Z$ is arithmetically Buchsbaum with $h^{1}(Q_{n}, {\mathcal I}_{Z}(r)) =1$ being the only nonzero intermediate cohomology.
			\item [(ii)]  If $n>3$ and $Z$ is arithmetically Buchbaum with $h^{1}(Q_{n}, {\mathcal I}_{Z}(r)) = 1$ being the only nonzero intermediate cohomology for $H^i(Q_{n},\mathcal{I}_Z)$, and $h^{2}(Q_{n}, N_{\sF}^{\vee}) = h^{n-1}(Q_{n}, N_{\sF}^{\vee}(-n+2)) = h^{n-1}(Q_{n},N_{\sF}^{\vee}(-r-n+1))= h^{n}(Q_{n},N_{\sF}^{\vee}(-r-n))=0,$ then $N_{\sF}^{\vee}$ is aCM.
		\end{itemize}
		
	\end{thm}

	Note that when $n=3,$ the item $(i)$ this Theorem corresponds to \cite[Theorem 4.1]{AMS} for the case $\iota_X = 3.$

	\hfill
	
	\subsection*{Acknowledgments}  We are grateful to Maurício Corrêa, Alan Muniz, and Yusuke Nakayama for interesting conversations. The authors also thank Laurent Manivel for providing many suggestions that helped improve the text and for encouraging them to extend Theorem \ref{reci_grass}  to general case. The second author is partially supported by the FAPEMIG RED-00133-21.
	
	\hfill

	%retirei buchsbaum bundles in quadrics e tem tbm em monads on projective space
	
	%\begin{Remark} \label{formula} \cite[Proposition 1.10]{H2}
	%For any holomorphic vector bundle $E$ of rank $r,$
	%$$\bigwedge^{k} E \simeq \bigwedge^{r-k}E^{\ast} \otimes \det E.$$
	%In particular, if $E$ is a rank $2$ reflexive sheaf, then
	%$ E^{\ast} = E \otimes (\det E)^{\ast}. $
	%\end{Remark}

	%%%%%%%%%%%%%%%%%%%%%%%%%%%%%%%%%%%%%%%%%%%%%%%%%%%%%%%%%%
	
	\section{Preliminaries} \label{section:preliminaries}
	
	%%%%%%%%%%%%%%%%%%%%%%%%%%%%%%%%%%%%%%%%%%%%%%%%%%%%%%%%%%
	In this Section we gather together some basic facts of the theory.
	
	%\stepcounter{thm}
	\subsection{Notations and Conventions}
	Throughout this paper we work over the field $\mathbb{C}$ of complex numbers.
	For denoting the dual of a vector space $V$ we will use $V^{\ast}$ and
	for denoting the dual of a vector bundle $F$, $F^{\vee}.$ We will not distinguish between vector bundle and locally free sheaf.
	Given a complex variety $X$, we  denote by $TX$ the tangent bundle $(\Omega_{X}^1)^{\vee}$ and
	to simplify the notation, given a distribution $\sF$ let us write $Z := \Sing(\sF)$ for its singular scheme.

	Assume that the Picard group of $X$ is $\mathbb{Z}.$ We will denote
	$E(t) = E \otimes_{\mathcal{O}_X} \mathcal{O}_X(t)$ for $t \in \mathbb{Z}$ when $E$ is a vector bundle on $X,$ and
	$\mathcal{O}_X(1)$ denotes its ample generator.
	
	If $F$ is a sheaf on $X,$ we denote by $h^i(X, F)$ the dimension of the complex vector space
	$H^i(X, F).$ As usual, $H_{\ast}^i (X, F) = \displaystyle \bigoplus_{t \in \Z} H^i(X, F(t)),$ for $i=0, \ldots,n.$

		\medskip
		%%%%%%%%%%%%%%%%%%%%%%%%%%%%%%%%%%%%%%%%%%%%%%%%%%%%%%%%%%
		%%%%%%%%%%%%%%%%%%%%%%%%%%%%%%%%%%%%%%%%%%%%%%%%%%%%%%%%%%
		
		%\stepcounter{subsection}
		\subsection{Holomorphic Distributions}\label{holomorphic distributions}
		
		Let us now recall facts about holomorphic distributions on complex projective varieties.
		For more details about distributions and foliations see \cite{MOJ,AMD,CJV,CMS} and the reference therein.

		\begin{Definition} \label{distributionTF}
			Let $X$ be a smooth complex manifold.
			%\begin{enumerate}
			%\item
			A \textit{codimension $k$ distribution} $\mathscr{F} $ on $X$ is given by an exact sequence
			\begin{equation}\label{eq:Dist}
				\mathscr{F}:\  0  \to T_{\sF} \stackrel{\phi}{ \to} TX \stackrel{\pi}{ \to} N_{\sF}  \to 0,
			\end{equation}
			where $T_{\sF}$ is a coherent sheaf of rank $r_{\sF}:=\dim(X)-k$, % $\phi : T_\sF \to TX$ is an injective morphism,
			and $N_{\sF}:= TX / \phi (T_{\sF})$ is a torsion free sheaf.
			% \item
			
			The sheaves $T_{\sF}$ and $N_{\sF}$ are called the \textit{tangent} and the \textit{normal} sheaves of $\sF$, respectively.
			The \textit{singular set} of the distribution $\sF$ is defined by 
			$$\Sing(\sF) = \{x \in X | (N_{\sF})_x  \; \mbox {is not a free} \; \mathcal{O}_{X,x}\mbox{-}\mbox{module}\}.$$
			
			The  \textit{conormal} sheaf  of $\sF$ is $N_{\sF}^{\vee}$.
			%\end{enumerate}
			
		\end{Definition}

		A distribution $\sF$ is said to be {\it locally free} if $T_{\sF}$ is a locally free sheaf. By definition, $\Sing(\sF)$ is the singular set of the sheaf $N_{\sF}$. It is a closed
		analytic subvariety of $X$ of codimension at least two, since by definition $N_{\sF}$ is torsion free.
		
		\begin{Definition}
			A \textit{foliation} is an integrable distribution, that is a distribution
			$$ \mathscr{F}:\  0  \to T_\sF \stackrel{\phi}{ \to} TX \stackrel{\pi}{ \to} N_{\sF}  \to 0 $$
			whose tangent sheaf is closed under the Lie bracket of vector fields, i.e. $[\phi(T_\sF),\phi(T_\sF)]\subset \phi(T_\sF)$.
			%A \emph{sub-foliation} of $T_\sF$ is an integrable sub-distribution.
		\end{Definition}

		Clearly, every distribution of codimension $\dim(X)-1$ is integrable. From now on, we will consider codimension one distributions on projective variety $X$ with Picard number one.  Thus, sequence (\ref{eq:Dist}) simplifies to
		\begin{equation}\label{eq:Dist cod1}
			\mathscr{F}:\  0  \to T_\sF \stackrel{\phi}{ \to} TX \stackrel{\pi}{ \to}
			\mathcal{I}_{Z/X}(r)  \to 0,
		\end{equation}

		\noindent where $T_\sF$ is a rank $n-1$ reflexive sheaf and $r$ is an integer number such that $r = c_1(TX) - c_1(T_{\sF})$. Observe that   $N_{\sF}=\mathcal{I}_{Z/X}(r)$ where $Z$ is the singular scheme of $\sF$; indeed, $N_{\sF}$ is by definition a torsion free sheaf, which in this case will have rank one, hence an ideal sheaf, which defines the singular locus of the distribution.
		
		A codimension one distribution on $X$ can also be represented by a
		section $\omega \in H^0 (X,\Omega^1_{X}(r))$, given by the dual of the morphism $\pi:TX \to N_{\sF}$. On the other hand, such section induces a sheaf map $\omega :\mathcal{O}_X \to \Omega^1_{X}(r)$. Taking duals, we get a cosection $$\omega^{\ast}:(\Omega^1_{X}(r))^{\ast}=TX(-r)\to \mathcal{O}_X$$
		whose image is the ideal sheaf $I_{Z/X}$ of the singular scheme. The kernel of $\omega^{\ast}$ is the tangent sheaf $\sF$ of the distribution twisted by $\mathcal{O}_{X}(-r)$.
		
		\begin{Remark}
			From this point of view, the integrability condition is equivalent to $\omega \wedge d\omega=0.$
		\end{Remark}
		%The 1-form $\omega$ can be written down in homogeneous coordinates
		%\[ \omega =\sum_{i=0}^3 A_i dz_i, \quad A_i\in H^0\mathcal{O}_X(d+1), \]
		%where $[z_0:z_1:z_2:z_3]$ homogeneous coordinates of $\P^3$; in addition, the coefficients $A_i$ must satisfy the condition
		%$$ i_R \omega = \sum_{i=0}^3 z_iA_i = 0 . $$

		Let $\mathcal{U}$ be the maximal subsheaf of $\mathcal{O}_{Z/X}$ of codimension $>2$, so that one has an exact sequence of the form
		%\stepcounter{thm}
		\begin{equation}\label{OZ->OC}
			0 \to \mathcal{U} \to \mathcal{O}_{Z/X} \to \mathcal{O}_{C/X} \to 0,
		\end{equation}
		where $C\subset X$ is a (possibly empty) subscheme of pure codimension $2.$ For more details, see \cite[section 2.1]{MOJ}, 
		
		The quotient sheaf is the structure sheaf of a subscheme $C\subset Z\subset X$ of pure dimension 1.
		
		\begin{Definition}
			If $Z$ is a 1-dimensional subscheme, then $Z$ has a maximal pure dimension 1
			subscheme $C$ defining a sequence
			%\stepcounter{thm}
			\begin{equation} \label{sequence I}
				0 \to \mathcal{I}_{Z} \to \mathcal{I}_{C} \to \mathcal{U} \to 0,
			\end{equation}
			where $\mathcal{U}$ is the maximal 0-dimensional subsheaf of $\mathcal{O}_{Z}.$
		\end{Definition}
		
		We finish this Subsection with the auxiliary result due Calvo-Andrade, Corrêa and Jardim.
		
		\begin{Lemma} \label{local-free-tg-sheaf}\cite[Lemma 2.1]{MOJ}
			The tangent sheaf of a codimension one distribution is locally free if and only if its singular locus has pure codimension $2.$
			%that is the singular locus has neither zero dimensional components, nor embedded points
		\end{Lemma}

		\medskip
		%%%%%%%%%%%%%%%%%%%%%%%%%%%%%%%%%%%%%%%%%%%%%%%%%%%%%%%%%%%%%%%%%%%%%%%%%%%%%%%%%%%%%%%%%%%%%%%%%%%%%%%%%%%%%%%%%%%%%%%%%%%%%%%%%%%%%%%%%%%%%%%%%%%%%%%%%%%%%%
		
		%\stepcounter{subsection}
		\subsection{The Eagon-Northcott complex}\label{eagon northcott}
		
		Let $X$ be a smooth projective variety, $\cA$ and $\cB$  locally free sheaves on $X$ of rank $a$ and $b,$ respectively, and $\eta : \cA \to \cB$ a generically surjective morphism. The induced map $\wedge^b \eta : \bigwedge^b \cA \to \det(B)$ corresponds to a global section $\omega_{\eta} \in H^0(X, \bigwedge^b \cA^{\vee} \otimes \det(B)).$
		
		\begin{Definition}
			The degeneracy scheme $\Sing(\eta)$ of the map $\eta : \cA \to \cB$ is the zero scheme of the associated global section $\omega_{\eta} \in H^0(X, \bigwedge^b \cA^{\vee} \otimes \det(B)).$
		\end{Definition}
		
		Suppose that $Z = \Sing(\eta)$ has pure codimension equal to $a -b + 1,$ i.e., $Z$ has pure expected dimension. Then the structure sheaf of $Z$ admits a special resolution, called the Eagon-Northcott complex
		(for more details see  \cite[Appendix 2.6.1]{Eisenbud}):
		
		\begin{center}
			\begin{equation*}
				0 \to \bigwedge^a \cA \otimes S_{a-b}(\cB^{\vee}) \otimes \det(\cB^{\vee}) \to \bigwedge^{a-1} \cA \otimes S_{a-b-1}(\cB^{\vee}) \otimes \det(\cB^{\vee}) \to \cdots 
			\end{equation*}
			\begin{equation*}
				\cdots \to  \bigwedge^{b+1} \cA \otimes \cB^{\vee}  \otimes \det(\cB^{\vee}) \to \bigwedge^{b} \cA  \otimes \det(\cB^{\vee}) \to \mathcal{I}_Z
				\to 0.
			\end{equation*}
		\end{center}
		
		\medskip
		%%%%%%%%%%%%%%%%%%%%%%%%%%%%%%%%%%%%%%%%%%%%%%%%%%%%%%%%%%
		%%%%%%%%%%%%%%%%%%%%%%%%%%%%%%%%%%%%%%%%%%%%%%%%%%%%%%%%%%
		%\stepcounter{thm}
		\subsection{aCM and aB schemes}
		
		A closed subscheme $Y \subset \mathbb{P}^n$ is \textit{arithmetically Cohen-Macaulay (aCM)} if its
		homogeneous coordinate ring $S(Y) = k[x_0, \ldots, x_n]/I(Y)$ is a Cohen-Macaulay
		ring.

		Equivalently, $Y$ is aCM if $H_{\ast}^p(\mathcal{O}_Y) = 0$ for $1 \leq p \leq \dim(Y)-1$
		and $H_{\ast}^1 (\mathcal{I}_Y) = 0$ (cf. \cite{CaH}).
		From the long exact sequence of cohomology associated to the short exact sequence
\begin{center}
	\begin{equation} \label{ideal sequence}
		0 \to {\mathcal I}_Y \to \mathcal{O}_{X} \to \mathcal{O}_Y \to
		0
	\end{equation}
\end{center}
		one also deduces that $Y$ is aCM if and only if $H_*^p({\mathcal
			I}_Y)=0$ for $1\leq p\leq\dim\,(Y)$.
		
		Similarly, a closed subscheme is \textit{arithmetically
			Buchsbaum (aB)} if its homogeneous coordinate ring is a Buchsbaum ring (see \cite{SV}).
		Clearly, every aCM scheme is arithmetically Buchsbaum, but the
		converse is not true: the union of two disjoint lines is
		arithmetically Buchsbaum, but not aCM.
		
		\medskip
	%%%%%%%%%%%%%%%%%%%%%%%%%%%%%%%%%%%%%%%%%%%%%%%%%%%%%%%%%%
	%%%%%%%%%%%%%%%%%%%%%%%%%%%%%%%%%%%%%%%%%%%%%%%%%%%%%%%%%%
	%\stepcounter{thm}
	\subsection{ Castelnuovo-Mumford Regularity}
	
Let $X$ be a projective scheme over a field $k,$ and let $m$ be an
integer.

\begin{Definition}
A coherent sheaf $E$ on $X$ is said to be $m$-regular if
$$H^i(E(m-i)) = 0,$$
for all $i > 0.$
\end{Definition}

The next theorem is in a more general setting in \cite[Lemma 1.8.3]{L}.

\medskip

\begin{thm} \label{CM-regularity}
If $E$ is $m$-regular, then the following holds:
\begin{enumerate}
	\item[1.] $E$ is $m'$-regular for all integers $m' \geq m;$
	\medskip
	\item[2.] $E(m)$ is globally generated.	
\end{enumerate}
\end{thm}

		\medskip
		%%%%%%%%%%%%%%%%%%%%%%%%%%%%%%%%%%%%%%%%%%%%%%%%%%%%%%%%%%%%%%%%%%%%%%%%%%%%%%%%%%%%%%%%%%%%%%%%%%%%%%%%%%%%%%%%%%%%%%%%%%%%%%%%%
		\section{Grassmann Manifolds} \label{grassm}
		%%%%%%%%%%%%%%%%%%%%%%%%%%%%%%%%%%%%%%%%%%%%%%%%%%%%%%%%%%%%%%%%%%%%%%%%%%%%%%%%%%%%%%%%%%%%%%%%%%%%%%%%%%%%%%%%%%%%%%%%%%%%%%%%%
		
		In this Section, $G$ will denote a Grassmann manifold $Gr(k,n),$ is the Grassmannian of linear subspaces $\mathbb{P}^{k}$ in the $\mathbb{P}^{n}$ on the field $\mathbb{C}$. In this case, the elements $\Lambda$ are seen as $k$ planes of $\p^n.$ Note that, in the particular case $k=0,$ i.e., the trivial Grassmannian $G$ is the projective space $\p^n.$ 
		The Grassmann variety of $k$ planes in $\p^n$ is naturally identified with the Grassmann variety of $(n-k-1)-$ planes in $\breve{\p^n}.$ By the duality, $Gr(k,n) \cong Gr(n-k-1,n).$	
	
	\medskip
	
	On $Gr(k,n)$ there is a canonical exact sequence	
	\begin{equation}\label{exactsequenceG}
		0 \to \mathcal{S} \to \mathcal{O}^{\oplus(n+1)}_{G} \to \mathcal{Q} \to 0,
	\end{equation}

\noindent where $\mathcal{S}$ has rank $k+1$ and is called the universal (tautological) bundle, $\mathcal{Q}$ has rank $n-k$ and is called the quotient bundle. Considering the isomorphism $Gr(k,n) \cong Gr(n-k-1,n),$ the canonical exact sequence on $Gr(n-k-1,n)$ is the dual sequence of (\ref{exactsequenceG}).
 
It is well known that $\mathcal{S}^{\vee}$ and $\mathcal{Q}$ are globally generated,  see for instance \cite{fujita}. We denote the dimension of the Grassmannian by $m = \dim \; Gr(k,n) = (k+1)(n-k).$ 
		\medskip

		\begin{thm} \cite{S} [Bott's formula for Grassmannian]\label{BottGr}
			Let $G = G(k+1,n+1)$ be the Grassmann manifold and let $n-k \geq k+1, \; m= dim G = (k+1)(n-k).$ Assume $1 \leq t \leq n.$ Then, $H^p(G,\Omega^q(t))=0$ if any of the following conditions is satisfied:
			\begin{enumerate}
				\item $(k+1)p \geq kq >0;$
				\item  $p>m-q;$
				\item  $q > m-k-1 = (k+1)(n-k-1),$ if $G \neq Gr(1,3);$
				\item $q \leq t$ and $p>0.$
			\end{enumerate}
			In particular, $H^p(G,\Omega^q(t))=0,$ if $p \geq \Big(\dfrac{k}{2k+1}\Big) \; m.$
		\end{thm}

		We start by proving  a few facts about the vanishing for the cohomology of twisted holomorphic forms for Grassmannian.

		\begin{Lemma} \label{vanish} Let $G$ be the Grassmann manifold $Gr(k,n)$ of $k$ planes in $\p^n,$ then:
			\begin{itemize}
				\item [(i)] $H^p_{\ast}(G, \Omega^{m-1}_G) = 0$ for $1 \leq p \leq m-2.$
				\item [(ii)] $H^p_{\ast}(G, \Omega^{m-i}_G) = 0$ for $1 \leq p < m-i$ and $2 \leq i \leq k.$ 
				\item [(iii)] $H^p_{\ast}(G, \cO_G) = 0$ for $1 \leq p \leq m-1.$
			\end{itemize}	
		\end{Lemma}

		\begin{proof}
			
			All the items in this Lemma come from a direct application of the Theorem \ref{BottGr} and the results contained in \cite{S}, using the Serre duality.
			
			%$H^p(X, \mathcal{O}(t)) = 0 \;\mbox{for all}\; t \in \Z$ and $1 \leq p \leq m-1.$ 	
		\end{proof}
		
		To present some results about distributions on Grassmann manifold we need to specify its definition as follows.  
		
		%Therefore, sequence (\ref{eq:Dist}) now reads as the short exact sequence
		
		%\begin{equation}\label{definiionD}
		%0  \to T_{\sF}  \to TX \to N_{\sF}  \to 0.
		%\end{equation}
		Take the $c$-th wedge product of the inclusion $N^{\vee}_{\sF} \subset \Omega_{G}^{1}$ the distribution can be given by a nonzero twisted differential $c$-form on $G$ locally decomposable
		
		$$\omega \in H^{0}(G, \Omega^{c}_{G} \otimes \det(N_{\sF})).$$
		
		The degree of the distribution $\sF$ is defined by degree of the zeros locus of the pullback $i^{\ast}\omega$, where
		
		$$i : G_{c} \to G,$$
		
		\noindent is the linear embedding of a generic $c$-plane. Then
		
		$$\deg (\sF) = \deg(Z(i^{\ast}\omega)) = \deg (c_{1}(N_{\sF})\otimes \Omega^{c}_{G_{c}}).$$
		
		Thus $\det(N_{\sF}) = \mathcal{O}_{G}(\deg(\sF) +c +1)$ and taken the determinant in the sequence (\ref{eq:Dist}) we get 
		
		$$\det(T_{\sF}) = \mathcal{O}_{G}(\dim(\sF) - \deg(\sF)).$$
		
		\begin{thm} Let $\sF$ be a codimension $c$ distribution on $G$ whose its singular locus $Z$ has the expected codimension $c+1.$
			If  $T_{\sF}$ splits as a sum of line bundles, then $Z$ is arithmetically Cohen Macaulay. 
		\end{thm}
		
		\begin{proof} 
			Let $m$ be the dimension of $G$ and let $c,d,g$ be, respectively, the codimension, degree and the rank of $\sF.$ Consider the Eagon-Northcott resolution of the structure sheaf of $Z,$ associated to the morphism $\eta^{*}: \Omega_G^1 \to T_{\sF}^{\vee}:$
			
			\begin{center}
				$0 \to \Omega^m_G \otimes S_c(T_{\sF}) (g-d) \stackrel{\varphi_c}{\to}
				\Omega^{m-1}_G \otimes S_{c-1}(T_{\sF}) (g-d)
				\stackrel{\varphi_{c-1}}{\to} \cdots$
			\end{center}
			\begin{center}
				$\cdots  \to  \Omega^{g+1}_G \otimes T_{\sF}(g-d) \stackrel{\varphi_1}{\to} \Omega^{g}_G (g-d) \stackrel{\varphi_0}{\to} \mathcal{I}_Z \to 0.$
			\end{center}
			
			Twisting by $\mathcal{O}_{G} (t),$ and break it down into the short exact sequences:
			\begin{center}
				$0 \to \Omega^m_G \otimes S_c(T_{\sF}) (g-d+t) \to \Omega^{m-1}_G \otimes S_{c-1}(T_{\sF}) (g-d+t) \to \Ker \; \varphi_{c-2}(t) \to 0$ 	
				$$\vdots$$
				$0 \to  \Ker \; \varphi_{c-i}(t)   \to \Omega^{m-i}_G \otimes S_{c-i}(T_{\sF}) (g-d+t) \to \Ker \; \varphi_{c-i-1} (t) \to 0$ 	
				$$\vdots$$
				$0 \to  \Ker \; \varphi_{0}(t)  \to \Omega^{g}_G (g-d+t) \to  \mathcal{I}_Z(t)  \to 0.$ 
			\end{center}
			
			If $T_{\sF}$ splits, so does its symmetric powers.
			%and hence $H^p(S_c(T_{\sF})) = 0 $ for $1 \leq p \leq n - 1.$ 
			Thus, $ \Omega^m_G \otimes S_c(T_{\sF}) \otimes \mathcal{O}_{G} (g-d+t)$ splits too.
			%$\Omega^m_X \otimes S_c(T_{\sF}) \otimes \mathcal{O}_{X} (n+1-r+t) = S_c(T_{\sF}) \otimes \mathcal{O}_X(t-r) = (\oplus \mathcal{O}_X(l_i)) \otimes \mathcal{O}_X(t-r) = \oplus \mathcal{O}_X(\tilde{l_i}).$
			By \cite[Proposition 1.4, p.325]{Ott2} term by term, we get $H^p(G,\Omega^m_G \otimes S_c(T_{\sF}) \otimes \mathcal{O}_{G} (g-d+t)) = 0$ for $1 \leq p \leq m-1$ and all $t \in \Z.$
			Therefore, 
			$$H^p(G,\Omega^{m-1}_G \otimes S_{c-1}(T_{\sF}) \otimes \mathcal{O}_{G} (g-d+t)) \simeq H^p(G,\Ker \; \varphi_{c-2}(t))$$
			for $1 \leq p \leq m-2$ and all $t \in \Z.$
			But, by item $(i)$ of Lemma \ref{vanish}, $H^p_{\ast}(G,\Omega^{m-1}_G \otimes S_{c-1}(T_{\sF}) \otimes \mathcal{O}_{G} (g-d+t)) = 0$ for $1 \leq p \leq m-2.$ It follows that $H^p_{\ast}(G,\Ker \; \varphi_{c-2}) = 0$ for $1 \leq p \leq m-2.$ Proceeding in the same way and using the item $(ii)$ of Lemma \ref{vanish}, we have that $H^p_{\ast}(G,\Omega^{m-i}_G \otimes S_{c-i}(T_{\sF}) \otimes \mathcal{O}_{G} (g-d+t)) = 0$ for $1 \leq p < m-i$ implies that $H^p_{\ast}(G,\Ker \; \varphi_{c-i-1}(t))=0$ for $1 \leq p < m-i.$ By looking at the next sequence, we see that $H^p(G,\Ker \; \varphi_{c-i}(t))=0$ for $1 \leq p < m-i$  and all $t\in \Z.$  So, looking at the last sequence, since $H^p_{\ast}(G,\Ker \; \varphi_{0}(t))=0$ we get $H^p_{\ast}(G,\mathcal{I}_Z) = 0$ for $1 \leq p \leq g-1.$ Therefore, $Z $ is arithmetically Cohen-Macaulay.
			
		\end{proof}

		The converse of the above result now follows. For it we need of a generalized Horrocks criterion to Grassmannian giving cohomological splitting conditions for vector bundles on Grassmannian split as a direct sum of line bundles, see \cite[Theorem 2.1, p. 326]{Ott2}.

		\bigskip
		
In this part we describe some little notations and properties about Schur functors to get more understand of vector bundles on $Gr(k,n)$, see \cite{Way} and the references therein for more details. We say that $\lambda = (\lambda_{1}, \ldots, \lambda_{m})$ is a partition of $d \in \mathbb{Z}$ if $\lambda_{1}+ \cdots + \lambda_{m} =d$. In this sense, to a partition $\lambda$ is associated a Young diagram if $\lambda_{1}\geq \cdots \geq \lambda_{m} \geq 0.$

If denote $\mathbb{S}_{\lambda}V$ by the Schur functor of a decreasing sequence of integer numbers $\lambda = (\lambda_{1}, \ldots, \lambda_{m})$ applied to the vector space $V$. We list some properties:

\begin{Proposition} For a vector space $V$ of dimension $n$ and with the notation above we have \\

\begin{itemize}

\item[a)] (Symmetric power) $\mathbb{S}_{(d)}V =S^{d}V$; \\

\item[b)] (Exterior power) $\displaystyle\mathbb{S}_{\underbrace{(1,\ldots,1)}_{d}}V = \wedge^{d}V$; \\

\item[c)] (Littlewood-Richardson) If $\lambda$ and $\mu$ are partitions of $d$ and $m$ respectively, then

$$\mathbb{S}_{\lambda} V \otimes \mathbb{S}_{\mu}V = \oplus_{\nu}N_{\lambda \mu \nu} \mathbb{S}_{\nu}V,$$

\noindent where $\nu$ is a partition of $|\nu|=d+m$;

\item[d)] $\mathbb{S}_{(\lambda_{1},\ldots,\lambda_{n})} V = \mathbb{S}_{(-\lambda_{n},\ldots,-\lambda_{1})}V^{\ast}.$

\end{itemize}
\end{Proposition}

In order to apply the Borel-Weil-Bott Theorem to vanish certain cohomology groups we need of some notations. If we use above properties on the quotient and universal bundles of the Grassmannian we have more specific cases: \\

$\bullet \ \ \mathcal{O}(1) = \wedge^{n-k}Q = \mathbb{S}_{\underbrace{(1,\ldots,1)}_{n-k}}Q$; \\

$\bullet \ \ \mathcal{O}(t) =\underbrace{\mathcal{O}(1)\otimes \cdots \otimes \mathcal{O}(1)}_{t} =\wedge^{n-k}Q \otimes \cdots \otimes \wedge^{n-k}Q = \mathbb{S}_{\underbrace{(t,\ldots,t)}_{n-k}}Q;$ \\

$\bullet \ \ \mathcal{O}(1) = \wedge^{k+1}S = \mathbb{S}_{\underbrace{(1,\ldots,1)}_{k+1}}S$; \\

$\bullet \ \ \mathcal{O}(t) =\underbrace{\mathcal{O}(1)\otimes \cdots \otimes \mathcal{O}(1)}_{t} =\wedge^{k+1}S \otimes \cdots \otimes \wedge^{k+1}S = \mathbb{S}_{\underbrace{(t,\ldots,t)}_{k+1}}S.$ \\

Now, we quote a version of the Borel-Weil-Bott Theorem for $Gr(k,n)$, see \cite[Corollary 4.19]{Way}.

\begin{thm}(Borel-Weil-Bott Theorem)\label{BWB} For each vector bundle on $G=Gr(k,n)$ of the form

$$ \mathbb{S}_{\beta}S^{\vee} \otimes \mathbb{S}_{\gamma}Q^{\vee} \ or \ (\mathbb{S}_{\gamma}Q \otimes \mathbb{S}_{\beta}S),$$ 

\noindent one of two mutually exclusive possibilities occurs for $\alpha = (\beta, \gamma) \ or \ \alpha = (\gamma, \beta):$ \\

\begin{enumerate}

\item[1)] There exist a permutation $\sigma \in \Sigma_{m}, \ \sigma \neq 1$ such that $\sigma.\alpha = \alpha$, then

$$H^{i}(G,\mathbb{S}_{\gamma}Q \otimes \mathbb{S}_{\beta}S) = 0, \ \ for \ \ i\geq 0.$$

\noindent \begin{Remark} This condition is equivalent $\alpha$ has repetitions entries, see Remark 2.2, in \cite{GLS}. 
\end{Remark}

\item[2)] There exists a unique permutation $\sigma \in \Sigma_{m}$ such that $\sigma.\alpha = \eta$ is a partition nonincreasing. In this case we have

$$ H^{i}(G,\mathbb{S}_{\gamma}Q \otimes \mathbb{S}_{\beta}S) =\left\{\begin{array}{clccc}
\neq 0,  & i=l(\sigma) \\ \\

0, & i \neq l(\sigma),
\end{array}\right.$$

\noindent where $l(\sigma)$ is the length of the permutation (number of inversions). 

\begin{Remark} For a precisely value fo these groups when it is nonzero, see \cite{Way}.
\end{Remark}

\end{enumerate}

\end{thm}

Consider $i_{1},\ldots, i_{s}$ integer numbers such that $0\leq i_{1},\ldots, i_{s} \leq n-k$, then

\begin{equation}\label{eq22}
\wedge^{i_{1}} Q^{\vee}\otimes \cdots \otimes \wedge^{i_{s}} Q^{\vee} =\oplus_{\nu} \mathbb{S}_{\nu} Q^{\vee},
\end{equation}

\noindent where $\nu= (\nu_{1},\ldots ,\nu_{n-k})$ is a partition of $i_{1}+ \cdots + i_{s}$ with $\nu_{1}\geq\cdots \geq\nu_{n-k}.$

We have $TG = Q\otimes S^{\vee}.$ Then we first tensor (\ref{eq22}) by $Q$,

\begin{equation}\label{eq33}
\mathbb{S}_{\nu}Q^{\vee}\otimes Q = \mathbb{S}_{\nu}Q^{\vee}\otimes \mathbb{S}_{(\underbrace{0,\ldots,0}_{n-k-1},-1)}Q^{\vee} = \oplus_{\mu} \mathbb{S}_{\mu} Q^{\vee},
\end{equation}

\noindent where $\mu = (\mu_{1},\ldots,\mu_{n-k})$ is obtained by subtracting 1 to some entry of $\nu$. We observe that

$$\mathcal{O}(t) = \mathcal{O}(1)\otimes \cdots \otimes \mathcal{O}(1)= \wedge^{n-k}Q\otimes \cdots \otimes\wedge^{n-k}Q = \mathbb{S}_{(\underbrace{t,\ldots,t}_{n-k})}Q, $$

\noindent and
$$S^{\vee} = \mathbb{S}_{(1,\underbrace{0,\ldots,0}_{k})} S^{\vee} = \mathbb{S}_{(0,\ldots,0,-1)}S.$$

Using it in (\ref{eq33}) we have

$$\begin{array}{rc}
\mathbb{S}_{\mu}Q^{\vee}\otimes S^{\vee}\otimes \mathcal{O}(t) =& \mathbb{S}_{\mu}Q^{\vee}\otimes \mathbb{S}_{(t,\ldots,t)}Q \otimes \mathbb{S}_{(0,...,0,-1)}S \\ \\

=& \mathbb{S}_{(t-\mu_{n-k},\ldots,t-\mu_{1})}Q\otimes \mathbb{S}_{(0,\ldots,0,-1)}S.

\end{array}$$

Then in order to apply the Theorem \ref{BWB} we consider the sequence
$$\alpha = (t-\mu_{n-k},\ldots,t-\mu_{1},\underbrace{0,\ldots,0}_{k},-1),$$
\noindent and subtracting consecutive integers

$$\rho = (n-k,\ldots,2,1,0,-1,-2,\ldots,-k+1,-k)$$

\noindent we have
$$(t +n-k-\mu_{n-k},\ldots,t+2-\mu_{2},t+1-\mu_{1},0,-1,-2,\ldots,-k+1,-k-1).$$

We analyze possible values of $t$ for this sequence and, using the Theorem \ref{BWB}, we obtain the following vanishing result. 

\vspace{7 cm}

\begin{thm}\label{themMain} With above notations on $G=Gr(k,n)$ and dimension $m=(k+1)(n-k)$, we have \\

\begin{enumerate}

\item[a)] If $t = \mu_{l}-l+j$ with $1\leq l \leq n-k, \ \ -k-1\leq j \leq 0$ and $j\neq -k$ the above sequence has repeated terms, then

%$$H^{i}(G, \wedge^{i_{1}} Q^{\vee}\otimes \cdots \otimes \wedge^{i_{s}} Q^{\vee}\otimes TG(t))=\left\{\begin{array}{clccc}
%\neq 0  & i=0 \\ \\
%
%0 & i>0
%\end{array}\right.$$

$$H^{i}(G, \wedge^{i_{1}} Q^{\vee}\otimes \cdots \otimes \wedge^{i_{s}} Q^{\vee}\otimes TG(t))=0 \ \ for \ \  i\geq 0,$$

\item[b)] If $t>\mu_{1}-1$ we have a decreasing sequence without repetitions, then

$$H^{i}(G, \wedge^{i_{1}} Q^{\vee}\otimes \cdots \otimes \wedge^{i_{s}} Q^{\vee}\otimes TG(t))=\left\{\begin{array}{clccc}
\neq 0,  & i=0 \\ \\

0, & i>0,
\end{array}\right.$$

\item[c)] If $\mu_{j}-j <t<\mu_{j-1}-k-j$ with $2\leq j \leq n-k$ we need to apply $(j-1)(k+1)$ inversions, then

$$H^{i}(G, \wedge^{i_{1}} Q^{\vee}\otimes \cdots \otimes \wedge^{i_{s}} Q^{\vee}\otimes TG(t))=\left\{\begin{array}{clccc}
\neq 0,  & i=(j-1)(k+1) \\ \\

0, & i\neq (j-1)(k+1),
\end{array}\right.$$

\item[d)] If $t<\mu_{n-k}-(n+1)$ we need to apply $(n-k)(k+1)=m$ inversions, then

$$H^{i}(G, \wedge^{i_{1}} Q^{\vee}\otimes \cdots \otimes \wedge^{i_{s}} Q^{\vee}\otimes TG(t))=\left\{\begin{array}{clccc}
\neq 0,  & i=m \\ \\

0, & i\neq m,
\end{array}\right.$$

\item[e)] If $t= \mu_{j}-j-k, \ \ 1\leq j \leq n-k$ with $\mu_{j}-\mu_{j+1} >k-1$ we need to apply $(j-1)(k+1)+k$ inversions, then

$$H^{i}(G, \wedge^{i_{1}} Q^{\vee}\otimes \cdots \otimes \wedge^{i_{s}} Q^{\vee}\otimes TG(t))=\left\{\begin{array}{clccc}
\neq 0,  & i=(j-1)(k+1)+k \\ \\

0, & i\neq (j-1)(k+1)+k.
\end{array}\right.$$

\end{enumerate}

\end{thm}

\begin{Corollary}\label{corol} For any partition $\eta = (\eta_{1}, \ldots, \eta_{n-k})$ of $i_{1} + \cdots + i_{s}$ we have

$$H^{0}(G,\wedge^{i_{1}} Q^{\vee}\otimes \cdots \otimes \wedge^{i_{s}} Q^{\vee} \otimes \mathcal{O}_{G}(t+r)) =\left\{\begin{array}{clccc}
\neq 0,  & t>\eta_{1}-1-r \\ \\

0, & t\leq\eta_{1}-1-r.
\end{array}\right.$$ \\

\end{Corollary}

\begin{proof} Given $i_{1},\ldots, i_{s}$ integer numbers such that $0\leq i_{1},\ldots, i_{s} \leq n-k$, we have the decomposition

\begin{equation}\label{eq3}
\wedge^{i_{1}} Q^{\vee}\otimes \cdots \otimes \wedge^{i_{s}} Q^{\vee} =\oplus_{\eta} \mathbb{S}_{\eta} Q^{\vee},
\end{equation}

\noindent where $\eta= (\eta_{1},\ldots ,\eta_{n-k})$ is a partition of $i_{1}+ \cdots + i_{s}$ with $\eta_{1}\geq\cdots \geq\eta_{n-k}.$ We can write

$$\mathcal{O}_{G}(t+r) = \mathbb{S}_{(t+r, \ldots, t+r)}Q.$$

Then
$$\begin{array}{rlc}
\wedge^{i_{1}} Q^{\vee}\otimes \cdots \otimes \wedge^{i_{s}} Q^{\vee}\otimes \mathcal{O}_{G}(t+r) =& \oplus_{\eta} \mathbb{S}_{\eta} Q^{\vee} \otimes \mathbb{S}_{(t+r, \ldots, t+r)}Q \\ \\

=& \oplus_{\eta} \mathbb{S}_{\underbrace{(t+r-\eta_{n-k},\ldots , t+r-\eta_{1})}_{n-k}} Q.

\end{array}$$

In order to apply the Theorem \ref{BWB} we get the sequence
$$(t+r-\eta_{n-k},\ldots , t+r-\eta_{1}, \underbrace{0,\ldots ,0}_{k+1}),$$

\noindent and subtracting the consecutive integers

$$\rho = (n-k, \ldots, 1,0,-1, \ldots ,-k)$$

\noindent we work with
$$(t+r +n-k-\eta_{n-k},\ldots , t+r+1-\eta_{1}, 0,-1,\ldots,-k).$$

To finish we study the values of $t$ we have repetitions, decreasing sequence or if we need to make inversions.

\end{proof}

\begin{Proposition}\label{proAux} Let $Z$ be a subset of $G=Gr(k,n)$ of pure codimension $n-k$ such that \\ $Z \simeq Gr(k-1,n-1)$, then for $0\leq i_{1},\ldots, i_{s} \leq n-k$ and $0< i < m,$

$$H^{i}(G, \wedge^{i_{1}} Q^{\vee}\otimes \cdots \otimes \wedge^{i_{s}} Q^{\vee}\otimes \mathcal{I}_{Z}(t)) = 0, \ \ \mbox{for} \ \ t\neq \dfrac{-i+1}{k}, \ \ \mbox{and} \ \ t \in \mathbb{Z}.$$

\end{Proposition}

\begin{proof} In fact, we can consider $Z \simeq Gr(k-1,n-1)$ as the zero locus of a generic section $s$ of $S$ ($Q$ is a globally generated), then we have the Koszul complex of $s$, that is an exact sequence

$$0 \to \wedge^{n-k} Q^{\vee} \to \cdots \to \wedge^{2} Q^{\vee} \to Q^{\vee} \to \mathcal{I}_{Z} \to 0. $$

Tensoring it by $E(t)$, where $E = \wedge^{i_{1}}Q^{\vee}\otimes \cdots \otimes \wedge^{i_{s}} Q^{\vee}(t),$ and  $t \in \mathbb{Z}$

$$0 \to E \otimes \wedge^{n-k} Q^{\vee}(t) \to \cdots \to E \otimes \wedge^{2} Q^{\vee}(t) \to E \otimes Q^{\vee}(t) \to E \otimes \mathcal{I}_{Z}(t) \to 0. $$

%\noindent We use the notation $A_{p}(t) = E \otimes \wedge^{p} Q^{\vee}(t)$ for $1 \leq p \leq n-k.$

For $0< i < m$ we have using Serre duality

$$H^{i+p-1}(G, E \otimes \wedge^{p} Q^{\vee}(t)) = H^{m-i-p+1}(G, \wedge^{i_{1}}Q \otimes \cdots \wedge^{i_{s}}Q \otimes \wedge^{p}Q(-t-n-1))=0,$$

\noindent for $t \neq \dfrac{-i+1}{k},$ with $t\in \mathbb{Z}$, and $1 \leq p \leq n-k$ according to \cite[Lemma 1.3, p. 322]{Ott2}. Thus from \cite[Lemma 1.1(i), p. 318]{Ott2} implies $H^{i}(G, E\otimes \mathcal{I}_{Z}(t)) = 0$ for $0< i < m$ as desired.

\end{proof}

As a consequence of the Proposition \ref{proAux} we get

\begin{Corollary} $Z$ is aCM.

\end{Corollary}

\begin{proof} We consider the long exact sequence, after tensoring by $\mathcal{O}_{G}(t)$, $t \in \mathbb{Z}$

$$0 \to \wedge^{n-k} Q^{\vee}(t) \to \cdots \to \wedge^{2} Q^{\vee}(t) \to Q^{\vee}(t) \to \mathcal{I}_{Z}(t) \to 0. $$

From \cite[Lemma 1.3 and Lemma 1.1(i)]{Ott2} with $1\leq p \leq n-k,$ and using Serre duality

$$H^{i+p-1}(G,\wedge^{p}Q^{\vee}(t)) = H^{m-i-p+1}(G, \wedge^{p}Q(-t-n-1))=0.   $$

That implies $H^{i}(G, \mathcal{I}_{Z}(t)) = 0$ for $0< i < m.$ Therefore $Z$ is ACM.

\end{proof}

\begin{thm}\label{reci_grass} Let $\sF$ be a codimension one holomorphic distribution on $G = Gr(k,n), n>2,$ such that its singular locus $Z$ has pure codimension $n-k$. Given $i_{1},\ldots, i_{s}$ integer numbers such that $0\leq i_{1},\ldots, i_{s} \leq n-k$, and $\eta= (\eta_{1},\ldots ,\eta_{n-k})$ any partition of $i_{1}+ \cdots + i_{s}$ with $\eta_{1}\geq\cdots \geq\eta_{n-k}.$ Suppose that: \\ 
	
	\begin{enumerate}
		
		\item[a)] $H^{1}(G, \wedge^{i_{1}} Q^{\vee}\otimes \cdots \otimes \wedge^{i_{s}} Q^{\vee}\otimes T_{\sF}(t)) = 0;$ for $t>\eta_{1}-1-r$; \\
		
		\item[b)] If $\mu_{j}-j <t<\mu_{j-1}-k-j$ with $2\leq j \leq n-k$ suppose		
		
		$$H^{(j-1)(k+1)}(G, \wedge^{i_{1}} Q^{\vee}\otimes \cdots \otimes \wedge^{i_{s}} Q^{\vee}\otimes T_{\sF}(t)) = 0;$$ 
		
		\item[c)] If $t= \mu_{j}-j-k, \ \ 1\leq j \leq n-k$ with $\mu_{j}-\mu_{j+1} >k-1$ suppose
		
		$$H^{(j-1)(k+1)+k}(G, \wedge^{i_{1}} Q^{\vee}\otimes \cdots \otimes \wedge^{i_{s}} Q^{\vee}\otimes T_{\sF}(t)) = 0;$$
		 
		\item[d)] $H^{i}(G, \wedge^{i_{1}} Q^{\vee}\otimes \cdots \otimes \wedge^{i_{s}} Q^{\vee}\otimes T_{\sF}(t)) = 0$ for $t = \dfrac{-i+2-rk}{k}.$ \\
		
\noindent Then $T_{\sF}$ splits as a direct sum of line bundles, where $r$ is the degree of the distribution.
\end{enumerate}
\end{thm}

\begin{proof} We consider the exact sequence defining the holomorphic distribution after twisting by

$$\wedge^{i_{1}} Q^{\vee}\otimes \cdots \otimes \wedge^{i_{s}}Q^{\vee}(t),$$

\noindent for $0\leq i_{1},\ldots, i_{s} \leq n-k$ and $t \in \mathbb{Z}.$

$$0 \to \wedge^{i_{1}} Q^{\vee}\otimes \cdots \otimes \wedge^{i_{s}}Q^{\vee}\otimes T_{\sF}(t) \to \wedge^{i_{1}} Q^{\vee}\otimes \cdots \otimes \wedge^{i_{s}}Q^{\vee} \otimes TG(t) \to \wedge^{i_{1}} Q^{\vee}\otimes \cdots \otimes \wedge^{i_{s}}Q^{\vee}\otimes \mathcal{I}_{Z}(t+r) \to 0. $$

Now we take the cohomology groups

$$\begin{array}{rccc}

H^{i-1}(G,\wedge^{i_{1}} Q^{\vee}\otimes \cdots \otimes \wedge^{i_{s}}Q^{\vee}\otimes \mathcal{I}_{Z}(t+r)) \to & H^{i}(G,\wedge^{i_{1}} Q^{\vee}\otimes \cdots \otimes \wedge^{i_{s}}Q^{\vee}\otimes T_{\sF}(t)) \to \\ \\

\to & H^{i}(G, \wedge^{i_{1}} Q^{\vee}\otimes \cdots \otimes \wedge^{i_{s}}Q^{\vee} \otimes TG(t)).

\end{array}$$

For $i=1$ and $t\leq\eta_{1}-1-r$ we have $H^{0}(G,\wedge^{i_{1}} Q^{\vee}\otimes \cdots \otimes \wedge^{i_{s}} Q^{\vee} \otimes \mathcal{O}_{G}(r+t)) = 0$ by Corollary \ref{corol} and using the standard exact sequence

$$0 \to \mathcal{I}_{Z} \to \mathcal{O}_{G} \to \mathcal{O}_{Z} \to 0$$

\noindent after twisting by $\wedge^{i_{1}} Q^{\vee}\otimes \cdots \otimes \wedge^{i_{s}}Q^{\vee}(t+r)$ we get

\begin{equation}\label{eq4}
H^{0}(G,\wedge^{i_{1}} Q^{\vee}\otimes \cdots \otimes \wedge^{i_{s}}Q^{\vee}\otimes \mathcal{I}_{Z}(t+r)) =0, \ \ \ \ t\leq\eta_{1}-1-r.
\end{equation}

We observe that the Theorem \ref{themMain} implies that $H^{1}(G,\wedge^{i_{1}} Q^{\vee}\otimes \cdots \otimes \wedge^{i_{s}}Q^{\vee}\otimes TG(t))=0$ for all $t \in \mathbb{Z}$. Using hypothesis $a)$ and display (\ref{eq4}) we get

$$H^{1}(G,\wedge^{i_{1}} Q^{\vee}\otimes \cdots \otimes \wedge^{i_{s}}Q^{\vee}\otimes T_{\sF}(t)) =0 \ \ for \ \ all \ \ t \in \mathbb{Z}.$$

For $1<i<m$ the Proposition \ref{proAux} implies $H^{i-1}(G,\wedge^{i_{1}} Q^{\vee}\otimes \cdots \otimes \wedge^{i_{s}}Q^{\vee}\otimes \mathcal{I}_{Z}(t+r)) = 0 $ for $t \neq \dfrac{-i+2-rk}{k}.$ On the other hand $H^{i}(G, \wedge^{i_{1}} Q^{\vee}\otimes \cdots \otimes \wedge^{i_{s}}Q^{\vee} \otimes TG(t)) =0$ when satisfies the Theorem \ref{themMain}. The hypotheses complete the prove.

\end{proof}

	In the case where the codimension of the distribution is $nk-k^2$ with $k \geq 1, \; n \geq 3,$ the tangent bundle of this distribution coincides with a universal quotient bundle twisted about the Grassmannian as can be seen in the example below.
		
		\begin{Example}
			The tangent sheaf $T_{\sF}$ of a distribution of codimension $c=nk-k^2$  on $G=Gr(k,n)$ is the dual of the universal quotient bundle twisted by $\mathcal{O}_{G}(-t),$ for all $t \geq 0$. 
			Indeed, since the quotient bundle is globally generated, $\mathcal{Q}(t)$ is globally generated, for all $t \geq 0$. Consequently,
			$\mathcal{Q}(t) \otimes TG$ is also globally generated, since $TG$ is globally generated. 
			Now, to show that there is such distribution, we apply \cite[Theorem A.3]{MOJ} with $L =  \mathcal{O}_{G}.$ Finally, we get
			$$ 0 \to \mathcal{Q}^{\vee}(-t) \to TG \to N_{\sF} \to 0.$$
		\end{Example} 
		
		\medskip
		
		A homological criterion for connectedness of the singular scheme of codimension one distributions on $\p^3 \simeq Gr(0,3)$ was proved by O. Calvo-Andrade, M. Corr\^ea and M. Jardim in \cite{MOJ}. This criterion was extended for others smooth weighted projective complete intersection Fano threefold with Picard number one in \cite{AMS}, by A. Cavalcante, M. Corrêa and S. Marchesi. Now, we observe that the singular locus of a codimension one
		distribution $\sF$ on the Grassmannian $Gr(2,3) \simeq \breve{\mathbb{P}^{3}}$ is connected recovering the result on $\mathbb{P}^{3}$, see \cite{MOJ}.

		\begin{thm} \label{connected cod1 gr}
			Let $\sF$ be a distribution of codimension one on $X=Gr(2,3)$ and singular scheme $Z.$ If $h^2(T_\sF(-r))=0$ and $C \subset X, \; C  \neq \emptyset,$ then $Z$ is connected and of pure dimension 1, so that $T_\sF$ is locally free.
			%	Conversely, if $Z = C$ is connected, then $T_\sF$ is locally free and $h^2(T_\sF(-r))=0$.
			
		\end{thm}
		%We emphasize that the hypothesis about the dimension of Grassmannian is essential to guarantee the vanishing of cohomology groups $h^p(TX(t))=0,\;\; p=1,2,$ for all $t \in \Z.$
		\begin{proof}
			Consider the exact sequence (\ref{eq:Dist cod1}). Twisting it by $\mathcal{O}_{X}(-r)$ and passing to cohomology we obtain,
			$$ H^1 (X,TX(-r)) \to H^1(X,\mathcal{I}_Z) \to  H^2(X,T_{\sF}(-r)) \to H^2(X,TX(-r)).$$  
			From item $(i)$ of the Lemma \ref{vanish}, we get that $ H^1 (X,TX(-r))=0.$ If $h^2(X,T_\sF(-r))=0$, then $h^1(X,\mathcal{I}_{Z})=0.$ It follows from the standard sequence
			$$ 0 \to \mathcal{I}_{Z} \to \mathcal{O}_{X} \to \mathcal{O}_{Z} \to 0 $$
			that
			$$ H^0(X,\mathcal{O}_{X} )\to H^0(X,\mathcal{O}_{Z})\to 0, $$
			hence $h^0(X,\mathcal{O}_{Z})=1$. From the sequence (\ref{OZ->OC}), we get
			$$ 0 \to H^0(X,\mathcal{U}) \to H^0(X,\mathcal{O}_{Z}) \to H^0(X,\mathcal{O}_{C}) \to 0.$$
			Thus either $h^0(X,\mathcal{O}_{C})=1$, and $\mathcal{U}=0$ and $C$ is connected, or
			$\mathrm{length}(\mathcal{U})=1$ and $C$ is empty. This second possibility is not valid because by
			hypothesis $C \neq \emptyset.$
			It follows that $Z=C$ must be connect and of pure dimension $1,$ and thus, by Lemma \ref{local-free-tg-sheaf}, $T_\sF$ is locally free.
			
			%	Conversely, assume that $Z=C$ is connected. Thus $Z$ must be of pure dimension 1, and Lemma \ref{local-free-tg-sheaf} implies that $T_\sF$ is locally free. It also follows that $h^1(I_{Z})=0,$ using Serre duality and item $(iii)$ of Lemma \ref{vanish}. Since $h^2(TX(-r))=0$ by item $(i)$ of the Lemma \ref{vanish}, we conclude that $h^2(T_\sF(-r))=0$, as desired.	

		\end{proof}
		
		As an immediate consequence, we obtain the following corollary.

		\begin{Corollary}
			Let $\sF$ be a codimension one distribution with singular scheme $Z$ on $X=Gr(2,3)$ the following holds:

			\begin{itemize}
				\item [(i)] If the tangent sheaf splits as a sum of line bundles, then $Z$ is connected;
				\item [(ii)] If $\sF$ has locally free tangent sheaf and
				$T_\sF^{\vee}$ is ample, then $Z$ is connected.
			\end{itemize} 
		\end{Corollary}
		
		\begin{dem}
			For the first item just assume that $T_\sF = \mathcal{O}_X (r_1)\oplus \mathcal{O}_X(r_2).$ Then, considering $r = r_1+r_2,$ we have that $h^2(T_\sF(-r))=0.$ The result follows from Theorem \ref{connected cod1 gr}.
			
			For the second item, by using Serre duality and Griffiths Vanishing Theorem \cite[Section 7.3.A, pg 335]{L}, respectively, we obtain 
			\begin{eqnarray*}
				H^2(X,T_\sF(-r)) \simeq H^{1}(X,T_\sF^{\vee}(r)\otimes K_{X})&=& H^{1}(X,T_\sF^{\vee}(r-c_1(TX))\otimes \mathcal{O}_X(c_1(TX))\otimes K_{X})\\
				&=&  H^{1}(X,T_\sF^{\vee}\otimes \det(T_\sF^{\vee})\otimes \mathcal{O}_X(c_1(TX))\otimes K_{X}) = 0. 
			\end{eqnarray*}
			The result follows from Theorem \ref{connected cod1 gr}.
		\end{dem}

		\bigskip
		
		%%%%%%%%%%%%%%%%%%%%%%%%%%%%%%%%%%%%%%%%%%%%%%%%%%%%%%%%%%%%%%%%%%%%%%%%%%%%%%%%%%%%%%%%%%%%%%%%%%%%%%%%%%%%%%%%%%%%%%%%%%%%%%%%

		%%%%%%%%%%%%%%%%%%%%%%%%%%%%%%%%%%%%%%%%%%%%%%%%%%%%%%%%%%%%%%%%%%%%%%%%%%%%%%%%%%%%%%%%
		\section{Smooth Quadric Hypersurfaces $Q_{n}$} \label{quadric}
		%%%%%%%%%%%%%%%%%%%%%%%%%%%%%%%%%%%%%%%%%%%%%%%%%%%%%%%%%%%%%%%%%
		In this Section we define the smooth quadric hypersurface and present some properties and necessary information about it. In $\mathbb{P}^{n+1}$ we denote the smooth quadric hypersurface of dimension $n$ by $Q_{n}$ and recall this variety can by consider again as a compact homogeneous manifold of the form
		
		$$G /P,$$
		
		\noindent where $G = S(n+2, \mathbb{C})$, and $P$ is the maximal parabolic subgroup.
		
		On quadrics there is a notable class of vector bundles, which is the natural generalization of the universal bundle that called spinor bundles. Now, we recall the definition and some properties of spinor bundles on
		$Q_n,$ in different ways depending on parity. For more details see \cite{Ott1,Ott2}.

		In the odd dimension, i.e. $Q_{2k+1} (n=2k+1)$, we have an embedding
		
		$$s : Q_{2k+1} \to Gr(2^{k}-1, 2^{2k+1} -1),$$
		
		\noindent in the Grassmannian of $2^{k}-1$ subspace of $\mathbb{P}^{2k+1}$, see \cite[p. 304]{Ott1}. Then we define the spinor bundle of rank $2^{k}$ by the following pullback
		
		$$s^{\ast}U = S,$$
		
		\noindent where $U$ is the universal bundle of the Grassmannian.
		
		In the even case dimension, i.e., $Q_{2k}(n=2k)$ we have two embeddings as follows
		
		$$s^{'} = Q_{2k} \to Gr(2^{k-1}-1, 2^{k}-1); \ \ \ \ \ \  s^{''} = Q_{2k} \to Gr(2^{k-1}-1, 2^{k}-1).$$
		
		Thus, we define the spinor bundles on $Q_{2k}$ of rank $2^{k-1}$ in the same way
		
		$$S^{'} = s^{'\ast} U; \ \ \ \ \ S^{''} = s^{''\ast} U. $$

		\begin{thm}\cite[Theorem 2.1]{Ott1} The spinor bundles on $Q_{n}$ are stable.
			
		\end{thm}
		
		A property relevant of the spinor bundles on quadrics $Q_{n}$ is that they are not split and have no intermediate cohomology, i.e., given $S$ a spinor bundle on $Q_{n}$, then
		
		$$H^{i}(Q_{n}, S) = 0, \ \ \mbox{for} \ \ 0<i<n.$$
		
		In 1964, Horrocks stabilizes in \cite{Horro} that vector bundles on $\mathbb{P}^{n}$ without intermediate cohomology split as a direct sum of line bundles. This concept leaves us a more general definition:
		
		\begin{Definition} A vector bundle $E$ on a projective variety $X$ is called arithmetically Cohen Macaulay (aCM) if
			$$H^{i}(X, E) = 0, \ \ \mbox{for} \ \ 0<i<\dim X.$$
			
		\end{Definition}

		\begin{thm} \label{teo 2.7}\cite[Theorem 2.1]{Bu}
			Let $F$ be a vector bundle on $Q_n, n \geq 3.$ If $F$ is arithmetically
			Cohen-Macaulay, i.e., it has no intermediate cohomology, then $F$ is a direct sum
			of line bundles and twisted by spinor bundles.
		\end{thm}
		
		Now we have a Horrocks's criterion version for quadrics, but we need to add a further more condition
		
		\begin{thm} \label{criterium} \cite[Theorem 1.5]{Ott4} Let $E$ be a vector bundle on $Q_{n}$. Let $S$ be a spinor bundle. Then,
			
			$$
			\begin{array}{cccc}
				E = \oplus_{i=1}^{r}\mathcal{O}(a_{i}) & \Longleftrightarrow \left\{\begin{array}{llccc}
					
					H^{i}_{\ast}(E) = 0 \ \ \mbox{for} \ \ 0<i<n \\ \\
					
					H_{\ast}^{n-1}(E\otimes S) = 0.

				\end{array}\right.

			\end{array}$$
		\end{thm}
		
		For the scope of this work, we need of a Bott's formula for quadrics, which we find in \cite{S}, which it is showed a vanishing theorem for the cohomology of $\Omega_{Q_{n}}^q(t)$ for quadric hypersurfaces $Q_n$ in $\p^{n+1}$.
		
		\begin{thm}\cite[Bott's formula for Quadric]{S} \label{BottQ}
			Let $X$ be a nonsingular quadric hypersurface of dimension $n.$
			\begin{enumerate}
				\item If $-n+q \leq k \leq q$ and  $k \neq 0$ and $k \neq -n+2q,$ then $H^p(X, \Omega^q(k)) = 0$ for all p;
				\item $H^p(X, \Omega^q) \neq 0$ if and only if $p=q;$
				\item $H^p(X, \Omega^q(-n+2q)) \neq 0$ if and only if $p = n-q;$
				\item If $k > q,$ then $H^p(X, \Omega^q(k)) \neq 0$ if and only if $p=0;$
				\item If $k< -n+q,$ then $H^p(X, \Omega^q(k)) \neq 0$ if and only if $p=n.$
			\end{enumerate}
		\end{thm}

		\medskip

		As an alternative ways to calculate the dimension of the cohomology group of differential forms we have a result due to Flenner, in \cite{flenner81}.
		
		\begin{thm} \label{flenner} \cite[Satz 8.11]{flenner81} Let $X$ be a weighted complete intersection. Then,
			\begin{enumerate}
				\item[-] $h^q(X,\Omega_{X}^q) =1 $ for $0\le q\le n$, $q\neq \frac{n}{2}$.
				\item[-] $h^p\big(X,\Omega_{X}^q(t)\big) = 0$ in the following cases
				\begin{itemize}
					\item[-] $0<p<n$, $p+q\neq n$ and either $p\neq q$ or $t\neq 0$;
					\item[-] $p+q > n$ and $t>q-p$;
					\item[-] $p+q < n$ and $t<q-p$.
				\end{itemize}
			\end{enumerate}
		\end{thm}

		Now, using the above results and notations we show some characterizations of the tangent and cotangent sheaves of distribution on quadrics.
		% We can recall the definition of a codimension $k$  holomorphic distribution $\sF$ on a Smooth Quadric Hypersurface $Q_{n}$ as given by the short exact sequence
		
		%$$  0  \to T_{\sF}  \to TQ_{n} \to N_{\sF}  \to 0.$$

		%is the degree of the foliation, and we will denote it by $\deg(\sF) = r$, see \cite{AMD}. 
		
		%If $k=1$ the above sequence may be rewritten as
		
		%$$  0  \to T_{\sF}  \to TQ_{n} \to \mathcal{I}_{Q_{n}}(r)  \to 0,$$
		
		%\noindent where the positive integer number $r := c_{1}(TQ_{n})-c_{1}(T_{\sF})$ is the degree of the foliation, see \cite{AMS} to more details.

		\begin{thm}
			Let $\sF$ be a distribution of codimension one on $Q_{n}$. If $T_{\sF}$ is a direct sum of line bundles and spinor bundle twisted by some $\mathcal{O}_{Q_{n}}(t)$, then $Z$ is arithmetically Buchsbaum, with $h^1(Q_{n}, \mathcal{I}_{Z} (r-2)) = 1$ being the only nonzero intermediate cohomology for $H^i(Q_{n},\mathcal{I}_{Z}).$

		\end{thm}
		
		\begin{dem} Consider the sequence (\ref{eq:Dist cod1}) for $Q_{n}.$ Twisting it by $\mathcal{O}_{Q_{n}}(t),$ and passing to cohomology, we get:
			\begin{eqnarray} \label{cohom}
				\cdots \to  H^1(Q_{n},T_{\sF}(t)) & \to H^1(Q_{n}, T Q_{n}(t)) \to &H^1(Q_{n}, \mathcal{I}_{Z} (r+t)) \to \\ \nonumber
				\to H^2(Q_{n},T_{\sF}(t)) & \to H^2(Q_{n}, T Q_{n}(t)) \to & H^2(Q_{n}, \mathcal{I}_{Z} (r+t))\to \cdots \\ \nonumber
				\vdots & \vdots &  \;\;\; \vdots  \\ \nonumber
				\cdots \to H^{n-2}(Q_{n},T_{\sF}(t)) & \to H^{n-2}(Q_{n}, TQ_{n}(t)) \to  &H^{n-2}(Q_{n}, \mathcal{I}_{Z} (r+t))\to \\ 
				\nonumber
				\to H^{n-1}(Q_{n},T_{\sF}(t)) &\to H^{n-1}(Q_{n}, TQ_{n}(t))  \to & H^{n-1}(Q_{n}, \mathcal{I}_{Z} (r+t))\to  \cdots
			\end{eqnarray}

			Since $T_{\sF}$ is a direct sum of line bundles and spinor bundle, it has no intermediate cohomology, see Theorem \ref{teo 2.7}, i.e.,
			$$H^i_{\ast}(Q_{n},T_{\sF}) = 0,  \  \;  \; 1 \leq i \leq n-1,$$
			and it follows that 
			$$H^i(Q_{n}, TQ_{n}(t)) \simeq H^i(Q_{n}, \mathcal{I}_{Z} (r+t)), \;\; 1 \leq i \leq n-2.$$
			
			%By Theorem \ref{flenner}, $H^2(X, T X(t)) = H^3(X, T X(t)) =  0,$ for all $t \in \Z.$ From Bott's formula for quadrics (Theorem \ref{BottQ}), we have that $H^1(X, T X(t)) = 0$ for all $t \neq -2,$ i.e. $H^1(X, T X (-2)) \neq 0.$ Hence, $H^1(X, \mathcal{I}_{Z} (r-2)) \neq 0$ and we conclude that $Z$ is arithmetically Buchsbaum.
			
			By Theorem \ref{flenner}, $H^2(Q_{n}, TQ_{n}(t)) = \cdots = H^{n-2}(Q_{n}, TQ_{n}(t)) =  0.$ From Bott's formula for quadrics (Theorem \ref{BottQ}), we have that $H^1(Q_{n}, TQ_{n}(t)) = 0$ for all $t \neq -2,$ i.e. $H^1(Q_{n}, TQ_{n}(-2)) \neq 0.$ Hence, $H^1(Q_{n}, \mathcal{I}_{Z} (r-2)) \neq 0$ and we conclude that $Z$ is arithmetically Buchsbaum as desired.  
			
		\end{dem}
		
		We observe that \cite[Theorem 3.2]{AMS} corresponds to the case $n=3$ for the above Theorem.
		
		\medskip
		
	\begin{Lemma} \label{lemma aux}
		Let $\sF$ be a codimension one holomorphic distribution on $Q_{n}$, n>3, with singular set denoted by $Z = \Sing(\sF)$. If $Z$ is arithmetically Buchsbaum with $h^{1}(Q_{n}, \mathcal{I}_{Z}(r-2)) =1$ being the only nonzero intermediate cohomology for
		$H^{i}(Q_{n}, \mathcal{I}_{Z})$ and $h^n(T_{\sF}(-r-n)) = 0,$ então $T_{\sF}$ is $(-r)$-regular in the sense Castelnuovo-Mumford.
	\end{Lemma}

\begin{dem}  Twisting the sequence (\ref{eq:Dist cod1}) by $\cO_{Q_n}(-r-i)$ and taking cohomology, we get
$$ 	\cdots \to  H^{i-1}(Q_{n}, \mathcal{I}_{Z}(-i)) \to  H^{i}(Q_{n},T_{\sF}(-r-i))  \to H^i(Q_{n}, TQ_{n}(-r-i))\to \cdots$$ 
The term on the left vanishes for $i\geq 2,$ since $Z$ is aB and when $i=1$ this term vanishes using the sequence (\ref{ideal sequence}) and Bott's formula.
On the other hand, applying the Theorems \ref{BottQ} and \ref{flenner}, the term on the right vanishes for all $1 \leq i \leq n-1.$ As by hypothesis $h^n(T_{\sF}(-r-n)) = 0,$ we conclude that $T_{\sF}$ is $(-r)$-regular.
	
\end{dem}
		
		\begin{thm}\label{thm3.7} Let $\sF$ be a codimension one holomorphic distribution on $Q_{n}$, n>3, with singular set denoted by $Z = \Sing(\sF)$ such that the tangent sheaf $T_{\sF}$ is locally free. If $Z$ is arithmetically Buchsbaum with $h^{1}(Q_{n}, \mathcal{I}_{Z}(r-2)) =1$ being the only nonzero intermediate cohomology for
			$H^{i}(Q_{n}, \mathcal{I}_{Z})$ and $h^{2}(Q_{n}, T_{\sF}(-2))= h^n(T_{\sF}(-r-n)) =  h^{n-1}(Q_{n}, T_{\sF}(-n)) = 0$, then $T_{\sF}$ has no intermediate cohomology, i.e., $T_{\sF} = \oplus \mathcal{O}_{Q_{n}}(t_{i})\oplus S(t)$, where $S$ is a spinor bundle on $Q_{n}.$
			
		\end{thm}
		
		\begin{dem}  Consider the sequence (\ref{eq:Dist cod1}) for all $t \neq -2$ and the piece of the long exact sequence of cohomology
			$$\begin{array}{llllll}
				\cdots \hspace{-1.5mm}&\to H^{i-1}(Q_{n}, T_{\sF}(t))  \to  H^{i-1}(Q_{n}, TQ_{n}(t)) \to H^{i-1}(Q_{n},\mathcal{I}_{Z}(r+t)) \to  \\ \\
				&\to   H^{i}(Q_{n}, T_{\sF}(t))  \to  H^{i}(Q_{n}, TQ_{n}(t)) \to H^{i}(Q_{n},\mathcal{I}_{Z}(r+t)) \to \cdots
			\end{array}$$
			
			\noindent for $i=1, \ldots, n-1.$

			To vanish the intermediate cohomology group of $T_{\sF}$ we vanish the adjacent groups. For this goal and using the isomorphism $TQ_{n}(t) \simeq \Omega_{Q_{n}}^{n-1}(t+n)$ one has

			\begin{equation}\label{eq3.6}
				H^{i}(Q_{n},TQ_{n}(t)) = H^{i}(Q_{n}, \Omega_{Q_{n}}^{n-1}(t+n)) = \left\{\begin{array}{llll} 0, & i =1; & t \neq -2 \\ \\
					0, & i= 2, \ldots , n-2 \\ \\
					
					0, & i =n-1, & t\neq -n.
					
				\end{array}\right.
			\end{equation}
			
			We have that 
			\begin{center} $H^{0}(Q_{n},\mathcal{I}_{Z}(r+t)) = 0$ for $t < -r,$  since $ H^{0}(Q_{n}, \mathcal{O}_{Q_{n}}(r+t)) =0,$  \end{center} by Theorem \ref{flenner}. 
			Considering this vanishing together with the vanishing corresponding to the first line of (\ref{eq3.6}), we obtain that $H^1(T_{\sF}(t))=0$ for all $t \neq -2$ and $t<-r.$ By the Lemma \ref{lemma aux} and Theorem \ref{CM-regularity}, $H^1(T_{\sF}(t))=0$ for all $t \neq -2$ and $t \ge -r.$
			Note that, as $Z$ is aB,
			\begin{equation} \label{eqSD}
H^i(\mathcal{I}_{Z}(r+t)) =0, \; \mbox{for} \; i=1, \ldots, n-2.
			\end{equation}
			Considering the equality (\ref{eqSD}) together with the vanishing corresponding to the second line of (\ref{eq3.6}), we obtain that $H^i(Q_{n},T_{\sF}(t))=0$ for all $t \neq -2$ and $i=2, \ldots, n-2.$
			Considering the equality (\ref{eqSD}) together with the vanishing corresponding to the third line of (\ref{eq3.6}) and the hypothesis $H^{n-1}(Q_{n},T_{\sF} (-n))=0,$ we conclude that $H^{n-1}(Q_{n},T_{\sF} (t))=0,$ for all $t \neq -2.$ 
			%For all $t \neq -2, \;\; T_{\sF}$ is aCM.
			
			Now, consider the above long exact sequence of cohomology with $t=-2.$ Applying Serre duality and Bott formula, respectively, we get $H^0(Q_n, \mathcal{I}_{Z}(r-2)) = H^n (Q_{n}, \mathcal{O}_{Q_{n}}(-r+2-n)) = 0,$ for all $r \neq 2.$
			Moreover, by Theorem \ref{flenner}, $H^{i}(TQ_n(-2)) = 0$, for $i=2, \ldots, n-1$ and by apllying Serre duality, $H^i(\mathcal{I}_{Z}(r-2)) = H^{n-i} (\mathcal{O}_{Q_{n}}(-r+2-n)) =0, \; \mbox{for} \; i=2, \ldots, n-2.$ So, we have that $H^i(T_{\sF}(-2))=0$ for $i=3, \ldots, n-1.$ By hypothesis, we get $h^2(T_{\sF} (-2)) = 0$ and
			$h^1(Q_n, \mathcal{I}_{Z}(r-2)) = 1.$ Thus, from the exact sequence,
			\begin{equation}\label{eq3.5}
				0  \to H^{1}(Q_{n}, T_{\sF}(-2)) \to H^{1}(Q_{n}, TQ_{n}(-2)) \to H^{1}(Q_{n}, \mathcal{I}_{Z}(r-2)) \simeq \mathbb{C} \to 0,
			\end{equation}
			%Then, for
			%$t = -2,$ $T_{\sF}$ is aCM.
			Therefore, $T_{\sF}$ is aCM.

		\end{dem}

		%OBSERVAÇÃO: WE NEED TO CLARIFY WHAT IS THE HYPOTHESIS: " Z IS aB".
		%
		%
		%\begin{Definition} (GBT Chat) A scheme $Z$ on a smooth quadric hypersurface $Q_{n}$ is arithmetically Buchsbaum (aB) if its ideal sheaf $I_{Z}$ satisfies the Cohomology property:
		%
		%The higher intermediate cohomology groups $H^{i}(Q_{n}, I_{Z})$ are independent of the twist $H^{i}(Q_{n}, I_{Z}(t))$ for all $t$. That means, there exists a single nonzero cohomology group among $H^{i}(Q_{n}, I_{Z})$ for $0<i<n$ while the others vanish.
		%
		%\end{Definition}

		%\begin{Lemma}\label{lemma1} Let $Q_{n}$ be the smooth quadric hypersurface on $\mathbb{P}^{n+1}$ with $S$ its spinor bundle, then
		
		%$$H^{i}_{\ast} (Q_{n} , \mathcal{I} \otimes S (t)) = 0 \ \ \ 1 \leq i \leq n-1,$$
		
		%\noindent where $\mathcal{I}$ is some CM sheaf on $Q_{n}.$
		
		%\end{Lemma}
		
		%\begin{dem} To prove, we use the Serre duality and the fact that the spinor bundle $S$ has no intermediate cohomology.
		%\end{dem}
		
		By using the splitting criteria (Theorem \ref{criterium}), we improve the classification on Theorem \ref{thm3.7} to check when the tangent sheaf of a distribution is only split.

		\begin{Proposition} With the same condition on Theorem \ref{thm3.7} on has the following
			
			\begin{enumerate}
				
				\item[a)] If $n=4$ and $H^{3}(Q_{4}, T_{\sF} \otimes S(t)) = 0$ for $t \leq -3,$ then $T_{\sF}$ splits; \\
				
				\item[b)] If $n=5$ then either $T_{\sF}$ splits or it is the spinor bundle; \\
				
				\item[c)] If $n=6$ and $H^{5}(Q_{6}, T_{\sF} \otimes S(t)) = 0$ for $t \leq -5,$ then $T_{\sF}$ splits; \\
				
				\item[d)] If $n>6$ then $T_{\sF}$ splits. \\
				
			\end{enumerate}
		\end{Proposition}
		
		\begin{dem} We will only show the assertion $a),$ since the $c)$ is analogous. The assertion $b)$ is the same of the Theorem \ref{thm3.7} and assertion $d)$ follows by the rank condition. The Theorem \ref{thm3.7} shows that $T_{\sF}$ has no intermediate cohomology, thus for to conclude, we need to prove the last condition $H^{n-1}_{\ast}(Q_{n}, T_{\sF} \otimes S) =0$ in Theorem \ref{criterium}. In $Q_{4},$ a spinor bundle has rank 2. We consider the exact sequence defining the distribution and the the specific cohomology sequence
			
			\[
			\resizebox{\textwidth}{!}{$
				\cdots \to H^{2}(Q_{4}, \mathcal{I}_{Z}\otimes S(r+t))  \to  H^{3}(Q_{4}, T_{\sF}\otimes S(t))
				\to H^{3}(Q_{4},TQ_{4}\otimes S(t)) \to H^{3}(Q_{4}, \mathcal{I}_{Z}\otimes S(r+t)) \to \cdots 
				$}
			\]
			
			To see that $H^{2}(Q_{4}, \mathcal{I}_{Z}\otimes S(r+t)) = 0,$ we tensorize the resolution of $S$, see \cite[Theorem 1.5, p.6]{Yo}
			$$0 \to \mathcal{O}_{\mathbb{P}^{5}}(-1)^{\oplus a_{1}} \to \mathcal{O}_{\mathbb{P}^{5}}^{\oplus a_{0}} \to S \to 0,$$
			
			\noindent by $\mathcal{I}_{Z}(r+t)$ and we use the hypothesis that $Z$ is aB.
			
			Moreover, $ H^{3}(Q_{4},TQ_{4}\otimes S(t)) = 0$ for $t > -3$ by Laytimi and Nagaraj, see \cite[Theorem 2.5 (i), p.470]{LayNa} and with the hypotheses, we conclude that $H^{3}_{\ast}(Q_{4}, T_{\sF}\otimes S) = 0.$ Therefore, $T_{\sF}$ splits as a direct sum of line bundles.
			
		\end{dem}
		
		\begin{Proposition}\label{TFsplit}
			Let $\sF$ be a codimension one distribution on $Q_n,$ with split tangent 	sheaf, i.e., $ T_\sF = \oplus_{i=1}^{n-1} \mathcal{O}_{Q_n}(a_i).$ Then, $a_i < 1.$
		\end{Proposition}

		\begin{dem}
			Indeed, we get
			\small{
				\begin{eqnarray*}
					T_\sF \hookrightarrow T_{Q_n} \in Hom (T_\sF, TQ_n)
					&\simeq& T_\sF^{\vee} \otimes TQ_n \\
					&\simeq& (\mathcal{O}_{Q_n}(-a_1) \otimes TQ_n) \oplus (\mathcal{O}_{Q_n}(-a_2) \otimes TQ_n) \oplus \cdots \oplus (\mathcal{O}_{Q_n}(-a_{n-1}) \otimes TQ_n).
			\end{eqnarray*} }
			Thus, we have $$H^0 (Q_{n}\oplus_{i=1}^{n-1} (\mathcal{O}_{Q_n}(-a_i) \otimes TQ_n)) = \oplus_{i=1}^{n-1} (H^0 (Q_{n}\mathcal{O}_{Q_n}(-a_i) \otimes TQ_n) ),$$
			and by Bott's fórmula for quadric, its has section when $a_i <1.$ 
			
		\end{dem}

		Now, under certain conditions, we expand the results above from codimension one holomorphic distribution to dimensions one, two, and three.
		
		\begin{thm}\label{dim2} Let $\sF$ be a dimension two holomorphic distribution on $Q_{n}$ ($n$ is even and $n \geq 4$) whose its singular set $Z$ has pure expected dimension. If $T_{\sF}$ splits, then $Z$ is aCM.
			
		\end{thm}
		
		\begin{proof} We use the Eagon-Northcott resolution, associated with the surjective morphism 
			$ \eta^{\ast} : \Omega^{1}_{Q_{n}} \to T_{\sF}^{\vee}$ defining the distribution.
			Tensoring it by $\mathcal{O}_{Q_{n}}(t)$, we get
			
			\begin{center}
				$0 \to \Omega^n_{Q_{n}} \otimes S_{n-2}(T_{\sF})\otimes \det T_{\sF}(t) \stackrel{\varphi_{n-2}}{\to}
				\Omega^{n-1}_{Q_{n}} \otimes S_{n-3}(T_{\sF})\otimes \det T_{\sF}(t)  \stackrel{\varphi_{n-3}}{\to} \cdots$
			\end{center}
			
			\begin{center}
				$\cdots  \to  \Omega^{3}_{Q_{n}} \otimes T_{\sF}\otimes \det T_{\sF}(t) \stackrel{\varphi_1}{\to} \Omega^{2}_{Q_{n}} \otimes \det T_{\sF}(t) \stackrel{\varphi_0}{\to} \mathcal{I}_Z(t) \to 0.$
			\end{center}

			Now, we break it down into some short exact sequences as follows:
			\begin{center}
				$0 \to \Omega^n_{Q_{n}} \otimes S_{n-2}(T_{\sF}) \otimes \det T_{\sF}(t) \to \Omega^{n-1}_{Q_{n}} \otimes S_{n-3}(T_{\sF}) \otimes \det T_{\sF}(t) \to \Ker \; \varphi_{n-4}(t) \to 0$ 	
				$$\vdots$$
				$0 \to  \Ker \; \varphi_{n-2-i}(t)  \to \Omega^{n-i}_{Q_{n}} \otimes S_{n-2-i}(T_{\sF})\otimes \det T_{\sF}(t) \to \Ker \; \varphi_{n-3-i} (t) \to 0$ 	
				$$\vdots$$
				$0 \to  \Ker \; \varphi_{0}(t)  \to \Omega^{2}_{Q_{n}} \otimes \det T_{\sF}(t) \to  \mathcal{I}_Z(t)  \to 0.$ 
			\end{center}
			
			%\noindent where in the middle sequence $i = 1, \ldots, n-2$. with $\Ker \; \varphi_{-1}(t) = \mathcal{I}_{Z}(t)$ and $\Ker \; \varphi_{n-3}(t) = \Omega^n_{Q_{n}} \otimes S_{n-2}(T_{\sF}) \otimes \det T_{\sF}(t).$ \textcolor{red}{não precisa dessa frase?}
			
			Take the long exact sequence of cohomology, for $i = 1, \ldots, n-2,$
			$$
			\begin{array}{cccc}
				
				\cdots \to  H^{n-1-i}(Q_{n}, \Omega^{n-i}\otimes S_{n-2-i}(T_{\sF})\otimes \det T_{\sF}(t)) \to H^{n-1-i}(Q_{n},\Ker \; \varphi_{n-3-i}(t)) \to & \\ \\
				
				\to H^{n-i}(Q_{n}, \Ker \; \varphi_{n-2-i}(t)) \to H^{n-i}(Q_{n},\Omega^{n-i}\otimes S_{n-2-i}(T_{\sF})\otimes \det T_{\sF}(t)) \to \cdots 
			\end{array}$$
			
			By Theorem \ref{flenner} and since $\sF$ splits so does its symmetric powers, we get 
			
			$$H^{n-1-i}(Q_{n}, \Omega^{n-i}\otimes S_{n-2-i}(T_{\sF})\otimes \det T_{\sF}(t)) =0 \ \ \mbox{for} \ \ i = 1, \ldots, n-2.$$ 
			This gives us an increasing inclusion

\begin{eqnarray*}
H^{1}(Q_{n}, \mathcal{I}_{Z}(t)) \subset H^{2}(Q_{n}, \Ker \; \varphi_{0}(t) ) \subset \cdots & \subset & H^{n-2}(Q_{n}, \Ker \; \varphi_{n-4}(t)) \\			 
		& \subset & H^{n-1}(Q_{n}, \Omega^{n}\otimes S_{n-2}(T_{\sF})\otimes \det(T_{\sF})(t)) =0, 
\end{eqnarray*}

			\noindent where last group is zero by again the result in Theorem \ref{flenner}. So $H^{1}_{\ast}(Q_{n},\mathcal{I}_{Z})=0$ and then $Z$ is aCM as desired.
			
		\end{proof}
		
		\begin{thm} Let $\sF$ be a dimension three holomorphic distribution on $Q_{5}$ whose its singular set $Z$ has pure expected dimension. If $T_{\sF}$ splits as $T_{\sF} = \mathcal{O}_{Q_{5}}(a_{1}) \oplus \mathcal{O}_{Q_{5}}(a_{2})\oplus \mathcal{O}_{Q_{5}}(a_{3})$, then $h^{2}(Q_{5}, \mathcal{I}_{Z}(1-a_{1}-a_{2}-a_{2})) =1$ being the only nonzero
			intermediate cohomology for $H^{i}(Q_{5},\mathcal{I}_{Z}).$
		\end{thm}
		
		\begin{proof} We consider again the Eagon-Northcott resolution associated the surjective morphism defining the distribution 
			$$\eta^{\ast} : \Omega_{Q_{5}}^{1} \to T_{\sF}^{\vee},$$
			
			\noindent tensoring by $\mathcal{O}_{Q_{5}}(t)$ and breaks it into short exact sequences \\
			
			$\begin{array}{ccc}

				0 \to \Omega^5_{Q_{5}} \otimes S_{2}(T_{\sF}) \otimes \det T_{\sF}(t) \to \Omega^{4}_{Q_{5}} \otimes T_{\sF} \otimes \det T_{\sF}(t) \to \Ker \varphi_{0}(t) \to 0 \\ \\
				
				0 \to \Ker \varphi_{0}(t) \to \Omega^3_{Q_{5}}\otimes \det(T_{\sF})(t) \to \mathcal{I}_{Z}(t) \to 0.
				
			\end{array}$ \\

			Since $H^{1}(Q_{5},\Omega^3_{Q_{5}}\otimes \det(T_{\sF})(t)) =0,$ for all $t \in \mathbb{Z}$ by Theorem \ref{flenner} we have  $H^{1}(Q_{5}, \mathcal{I}_{Z}(t)) \subset H^{2}(Q_{5}, \Ker \varphi_{0}(t)).$ We have again by Theorem \ref{flenner} that $H^{2}(Q_{5},\Omega^{4}_{Q_{5}} \otimes T_{\sF} \otimes \det T_{\sF}(t)) =0$ then $H^{2}(Q_{5}, \Ker \varphi_{0}(t)) \subset H^{3}(Q_{5}, \Omega^5_{Q_{5}} \otimes S_{2}(T_{\sF})\otimes\det(T_{\sF})(t)) =0.$ Thus $H^{2}_{\ast}(Q_{5},\Ker \varphi_{0})=H^{1}_{\ast}(Q_{5}, \mathcal{I}_{Z}) =0.$
			
			On the other hand as $H^{3}(Q_{5}, \Omega^{4}_{Q_{5}}\otimes T_{\sF}\otimes \det T_{\sF}(t)) = H^{4}(Q_{5}, \Omega^{5}_{Q_{5}}\otimes S_{2}(T_{\sF})\otimes \det T_{\sF}(t)) = 0$ we get $ H^{3}_{\ast}(Q_{5},\Ker \varphi_{0}(t)) = 0$. Then we have by Bott's formula for quadrics, and \cite[table 3, p.175]{S}

			$$\begin{array}{rlcc}
				
				h^{2}(Q_{5}, \mathcal{I}_{Z}(t)) = & h^{2}(Q_{5}, \Omega^{3}_{Q_{5}}\otimes \det T_{\sF}(t)) \\ \\

				= & h^{2}(Q_{5}, \Omega^{3}_{Q_{5}}(a_1 + a_2 + a_3 +t)) \\ \\
				
				= & \left\{\begin{array}{llll} 0, & t \neq 1-a_1 - a_2 - a_3  \\ \\
					
					1, & t = 1-a_1 - a_2 - a_3.
					
				\end{array}\right.
				
			\end{array}
			$$

			Therefore, $h^{2}(Q_{5}, \mathcal{I}_{Z}(1-a_{1}-a_{2}-a_{2})) =1$ being the only nonzero
			intermediate cohomology for $H^{i}(Q_{5},\mathcal{I}_{Z})$ as desired.
			
		\end{proof}

			Considering on the result in \cite[Theorem 11.8]{MOJ}, we produce an example of codimension one distribution on $Q_5,$ whose tangent sheaf is isomorphic to the spinor bundle up to twist.

			\begin{Example}
				%	A Spinor distribution $\sF$ on $Q_5$ is a codimension one distribution whose tangent sheaf $T_{\sF} \simeq S(-t),$ for all $t \geq 0$. 
				The tangent sheaf $T_{\sF}$ of a codimension one distribution on $Q_5,$ can be seen as the bundle spinor twisted by $\mathcal{O}_{Q_5}(-t),$ for all $t \geq 0$. 
				Indeed, since the dual of the spinor bundle  $S^{\vee}$ is globally generated, $S^{\vee}(t)$ is globally generated, for all $t \geq 0$. Consequently,
				$S^{\vee}(t) \otimes TQ_5$ is also globally generated, since $TQ_5$ is globally generated. 
				Now, to show that there is such codimension one distribution, we apply \cite[Theorem A.3]{MOJ} with $L =  \mathcal{O}_{Q_5}.$ Finally, we get
				$$ 0 \to S (-t) \to TQ_5 \to \mathcal{I}_Z (7+4t) \to 0.$$
				
			\end{Example}

			\begin{Remark} 
				The stability of $TQ_5$ implies that there are no injective morphisms $S(-t) \hookrightarrow TQ_5$ when $t \leq -2,$ since $\mu(S(-t)) = -\frac{1}{2}-t > 1 = \mu(TQ_5).$
				%	\item [(ii)] When $t=-1,$ note that $\Hom(S(1),TQ_5)\simeq H^0(S\otimes TQ_5)$. First, noting that $H^0(\mO_{Q_5}(1))\simeq H^0(\op6(1))$, we use the Euler exact sequence for $T \p^6$ restricted to $Q_5$, namely
				%	\begin{equation}\label{seq euler}
					%		0 \to   \mO_{Q_5}  \to H^0(\mO_{Q_5}(1))\otimes \mO_{Q_5}(1) \to T \to 0, 
					%	\end{equation}
				%	where $T:=T\p^6|_{Q_5}.$
				%	Twisting it by $S$ and taking the cohomology we obtain $H^0(T) \simeq H^0(\mO_{Q_5}(1))\otimes H^0 (S(1))$, since $h^0(S)=h^1(S)=0$. Next, 
				%	 twist the exact sequence
				%	\begin{equation}\label{normal-bundle}
					%		0 \to T{Q_5} \to T \to \mO_{Q_5}(2) \to 0, 
					%	\end{equation}
				%by $S$ to conclude that 
				%	\begin{equation}\label{dim hom s}
					%	h^0(S\otimes TQ_5) =  h^0(S \otimes T) - h^0(S(2)) = 8.    
					%\end{equation} 
					
					%\end{itemize}
				\end{Remark}

				%\textcolor{red}{O morfismo $S(1) \to TQ_5$ é injetivo e tem conucleo livre de torção? análogo p.17 do artigo foliation by curves}
				
				%
				%\begin{Proposition}
				%The nontrivial morphism $\phi: S(1) \to TQ_5$ is injective and has torsion-free cokernel.
				%\end{Proposition}
				%
				%\begin{proof}
				%	If a nontrivial morphism $\phi: S(1) \to TQ_5 $ is not injective, then $\ker \phi$ must be a nontrivial subsheaf of $S(1).$ However, $S(1)$ is stable, leading to a contraction and forcing $\ker \phi=0$.
				%\end{proof}
				
				\bigskip
				
				%
				%\textcolor{blue}{
					%	\begin{equation}\label{sequence I}
						%		0 \to \mathcal{U} \to \oz \to \mathcal{O}_C \to 0
						%	\end{equation}
					%	We show how to determine the number of connected components of $\Sing_1(\sF)$ for a distribution of codimension one $\sF$ on $Q_n.$ 
					%	\begin{thm} \label{connected cod1}
						%		Let $\sF$ be a distribution of codimension one on $Q_n,$ with  $\Sing_1(\sF):= C$ reduced. For $r \neq 2, \;\; h^2(T_{\sF}(-r))=c_3(T_{\sF})$ if and only if $\Sing_1(\sF)$ is connected.
						%	\end{thm}	
					%	\begin{proof} $Z=C??$ \\ \\
						%	$? n >3 ?$\\ \\
						%			Consider the exact sequence (\ref{eq:Dist Fano}) defining the distribution $\sF$, with $X = Q_n.$ Twisting it by $\mathcal{O}_{Q_n}(-r)$ and passing to cohomology we obtain,
						%			$$ H^1 (TQ_n(-r)) \to H^1(\sI_Z) \to  H^2(T_{\sF}(-r)) \to H^2(TQ_n(-r)).$$  
						%		Thus, for $r \neq 2, \;\; h^1(\sI_Z) =  h^2(T_{\sF}(-r)).$ From the standard sequence $0 \to \sI_Z \to \mathcal{O}_{Q_n} \to \oz \to  0$, we get $$h^1(\sI_Z) =  h^0(\oz) - h^0 (\mathcal{O}_{Q_n}) ?????????.$$
						%		 If $h^2(T_{\sF}(-r)) = c_3(T_{\sF}),$ follow from the last equality that $h^0(\mathcal{O}_C) = 1.$ Then, $\Sing_1(\sF)$ is connected.
						%		Conversely, assume that $\Sing_1(\sF)$ is connected, then $h^0(\mathcal{O}_C) = 1.$ It follows that $c_3(T_{\sF}) = h^1(\sI_Z) = h^2(T_{\sF}(-r)),$ as desired.
						%	\end{proof}
					%}
				%

				In this part of the paper we dear with conormal sheaf of a one-dimensional distribution and given condition for it to be aCM sheave. Definition \ref{distributionTF} provides an alternative way to define a foliation using a coherent subsheaf $ N_{\sF}^{\vee}$ of $\Omega^{1}_{Q_{n}}$ such that $ N_{\sF}^{\vee}$ is integrable and the quotient $\Omega_{Q_{n}}^{1} / N_{\sF}^{\vee}$ is torsion free. The codimension of $\sF$ is the generic rank of $N_{\sF}^{\vee}$ and let us denote $r = c_{1}(\Omega_{Q_{n}}^{1}) - c_{1}( N_{\sF}^{\vee}).$

				A. Cavalcante, M. Corrêa and S. Marchesi, showed in \cite{AMS} that if the conormal sheaf of a 
				foliation of dimension one on Fano 
				threefold splits, then its singular scheme is arithmetically Buchsbaum with 
				$h^1(\mathcal{I}_Z(r)) =1$ being the only nonzero intermediate cohomology \cite[Theorem 4.1]{AMS}.
				Based in this result, we extend it for any quadric hypersuface $Q_n, n>3.$

				\begin{thm} \label{conormal} Let $\sF$ be a one-dimensional holomorphic distribution on $Q_{n}.$ If the conormal sheaf $N_{\sF}^{\vee}$ is aCM, then $Z$ is arithmetically Buchsbaum with $h^{1}(Q_{n}, {\mathcal I}_{Z}(r)) =1$ being the only nonzero intermediate cohomology.
					
				\end{thm}
				
				\begin{proof} Consider the foliation $\sF$ given by the short exact sequence and tensorizing it $\mathcal{O}_{Q_{n}}(t)$, one has
					\begin{equation} \label{dist-conormal}
						0  \to N_{\sF}^{\vee}(t) \to \Omega_{Q_{n}}^{1}(t) \to  \mathcal{I}_{Z}(r+t) \to 0.
					\end{equation}
					
					%We pick the long exact sequence of cohomology for $i=0, \ldots, n$ (??, i =1 ???)
					
					Consider the piece of the long exact sequence for $i=1, \ldots, n-1.$ 
					
					\begin{equation*}
						\cdots \to	H^{i}(Q_{n}, N_{\sF}^{\vee}(t))  \to  H^{i}(Q_{n}, \Omega_{Q_{n}}^{1}(t)) \to H^{i}(Q_{n},{\mathcal I}_{Z}(r+t)) \to \cdots
					\end{equation*}
					
					%$$\begin{array}{llllll}
						%\cdots \hspace{-1.5mm}&\to H^{i-1}(Q_{n}, N_{\sF}^{\ast}(t))  \to  H^{i-1}(Q_{n}, \Omega_{Q_{n}}^{1}(t)) \to H^{i-1}(Q_{n},I_{Z}(r+t)) \to  \\ \\
						%&\to   H^{i}(Q_{n}, N_{\sF}^{\ast}(t))  \to  H^{i}(Q_{n}, \Omega_{Q_{n}}^{1}(t)) \to H^{i}(Q_{n},{\mathcal I}_{Z}(r+t)) \to \cdots
						%\end{array}$$

						Since $N_{\sF}^{\vee}$ is aCM, it has no intermediate cohomology, i.e.,
						$$ H^{1}_{\ast}(Q_{n}, N_{\sF}^{\vee}) = \cdots = H^{n-1}_{\ast}(Q_{n}, N_{\sF}^{\vee}) = 0.$$

						This implies, for $i =2, \ldots, n-2,$ the isomorphism $H^{i}(Q_{n}, \Omega_{Q_{n}}^{1}(t)) \simeq H^{i}(Q_{n}, {\mathcal I}_{Z}(r+t))$. By Theorem \ref{flenner}, $H^{i}(Q_{n}, \Omega_{Q_{n}}^{1}(t)) = 0,$ for $i =2,\ldots, n-2,$ and by Theorem \ref{BottQ}, $H^{1}(Q_{n}, \Omega_{Q_{n}}^{1}(t)) = 0,$ for $t \neq 0.$ Therefore, $1=h^{1}(Q_{n}, \Omega_{Q_{n}}^{1}) = h^{1}(Q_{n}, {\mathcal I}_{Z}(r))$ is the unique nonzero cohomology and we conclude that $Z$ is aB, as desired.
						
					\end{proof}

					We now provide the converse of Theorem \ref{conormal} by using the criterion \cite[Theorem 3.5, p.337]{Ott2}. Thus, our next theorem shows when conormal sheaf of a distribution on $Q_{n}, n>3,$ is isomorphic to a direct sum of line bundles and twisted spinor bundle.
	
\begin{Lemma} \label{lemma aux2}
	Let $\sF$ be one-dimensional holomorphic distribution on $Q_{n}$, n>3, with singular set denoted by $Z = \Sing(\sF)$. If $Z$ is arithmetically Buchsbaum with $h^{1}(Q_{n}, \mathcal{I}_{Z}(r)) =1$ being the only nonzero intermediate cohomology for
	$H^{i}(Q_{n}, \mathcal{I}_{Z})$ and $h^{n-1}(N_{\sF}^{\vee}(-r-n+1))=h^n(N_{\sF}^{\vee}(-r-n)) = 0,$ then $N_{\sF}^{\vee}$ is $(-r)$-regular in the sense Castelnuovo-Mumford.
\end{Lemma}

\begin{dem}  Considering the sequence (\ref{dist-conormal}), where $t = -r-i,$ and taking cohomology, we get
	$$ 	\cdots \to  H^{i-1}(Q_{n}, \mathcal{I}_{Z}(-i)) \to  H^{i}(Q_{n},N_{\sF}^{\vee}(-r-i))  \to H^i(Q_{n}, \Omega^1_{Q_{n}}(-r-i))\to \cdots$$ 
	The term on the left vanishes for $i>2,$ since $Z$ is aB and when $i=1$ this term vanishes using the sequence (\ref{ideal sequence}) and Bott's formula.
	On the other hand, applying the Theorems \ref{BottQ} and \ref{flenner}, the term on the right vanishes for all $1 \leq i \leq n-2.$ As by hypothesis $h^{n-1}(N_{\sF}^{\vee}(-r-n+1))=h^n(N_{\sF}^{\vee}(-r-n)) = 0,$ we conclude that $N_{\sF}^{\vee}$ is $(-r)$-regular.
\end{dem}					

\begin{thm} Let $\sF$ be an one-dimensional holomorphic distribution on $Q_{n}$. If $Z$ is arithmetically Buchbaum with $h^{1}(Q_{n}, {\mathcal I}_{Z}(r)) = 1$ being the only nonzero intermediate cohomology for $H^i(Q_{n},\mathcal{I}_Z)$, and $h^{2}(Q_{n}, N_{\sF}^{\vee}) = h^{n-1}(N_{\sF}^{\vee}(-r-n+1))= h^{n-1}(Q_{n}, N_{\sF}^{\vee}(-n+2)) = h^n(N_{\sF}^{\vee}(-r-n)) = 0,$ then $N_{\sF}^{\vee}$ is aCM.

						% i.e., it can be write as $N_{\sF}^{\ast} = \oplus \mathcal{O}(a_{i})\oplus S(t).$
						
					\end{thm}
					
					\begin{proof} In fact, as the above result, given a one-dimensional distribution $\sF$ we can consider the long exact sequence of cohomology
						
						$$\begin{array}{llllll}
							\cdots \hspace{-1.5mm}&\to H^{i-1}(Q_{n}, N_{\sF}^{\vee}(t))  \to  H^{i-1}(Q_{n}, \Omega_{Q_{n}}^{1}(t)) \to H^{i-1}(Q_{n},{\mathcal I}_{Z}(r+t)) \to  \\ \\
							&\to   H^{i}(Q_{n}, N_{\sF}^{\vee}(t))  \to  H^{i}(Q_{n}, \Omega_{Q_{n}}^{1}(t)) \to H^{i}(Q_{n},{\mathcal I}_{Z}(r+t)) \to \cdots
						\end{array}$$

By Theorem \ref{flenner}, $H^{1}(Q_{n}, \Omega_{Q_{n}}^{1}(t)) = 0$ for all $t\neq 0$	and
	\begin{center} $H^{0}(Q_{n},\mathcal{I}_{Z}(r+t)) = 0$ for $t < -r,$  since $ H^{0}(Q_{n}, \mathcal{O}_{Q_{n}}(r+t)) =0.$  \end{center}  
 By the Lemma \ref{lemma aux2} and Theorem \ref{CM-regularity}, $H^1(N_{\sF}^{\vee}(t))=0$ for $t \ge -r.$  Then, $H^{1}(Q_{n}, N_{\sF}^{\vee}(t)) = 0$ for all $t \neq 0$.
 
For $t=0$, using the hypothesis $h^{2}(Q_{n}, N_{\sF}^{\vee}) =0$ and Serre duality for to obtain the isomorphism $H^{0}(Q_{n}, {\mathcal I}_{Z}(r)) \simeq H^{n}(Q_{n}, \mathcal{O}_{Q_{n}}(-r-n)) =0$ for $r \neq 0,$ we have the short exact sequence,
						
						$$0 \to H^{1}(Q_{n}, N_{\sF}^{\vee})  \to  H^{1}(Q_{n}, \Omega_{Q_{n}}^{1}) \to H^{1}(Q_{n},{\mathcal I}_{Z}(r)) \to 0.$$
						
						\noindent Since $H^{1}(Q_{n}, {\mathcal I}_{Z}(r)) \simeq \mathbb{C}$ and $H^{1}(Q_{n}, \Omega_{Q_{n}}^{1}) \simeq \mathbb{C},$ we conclude $h^{1}(Q_{n}, N_{\sF}^{\vee}) =0.$  Then, $H^{1}_{\ast}(Q_{n}, N_{\sF}^{\vee}) = 0.$

						Now, note that $H^{1}(Q_{n}, {\mathcal I}_{Z}(r+t)) =0$ for $t \neq 0$ and by Theorem \ref{flenner}, $H^{2}(Q_{n},\Omega_{Q_{n}}^{1}(t)) = 0$ for all $t \in \mathbb{Z}$. So, together with the hypothesis $h^{2}(Q_{n}, N_{\sF}^{\vee}) =0$, we have $H^{2}_{\ast}(Q_{n}, N_{\sF}^{\vee})=0.$
						
						Next, noting that $H^{i-1}({\mathcal I}_{Z}(r+t)) =0$ for $i= 3, \ldots, n-2,$ and for all $t \in \mathbb{Z}$ and by Theorem \ref{flenner} $H^{i}(\Omega_{Q_{n}}^{1}(t)) = 0,$ for $i=3, \ldots, n-2,$ and for all $t \in \mathbb{Z},$ it follows that $H^{i}_{\ast}(N_{\sF}^{\vee}) = 0,$ for $i=3,\ldots, n-2$.
						
						To finish, $H^{n-2}({Q_{n},\mathcal I}_{Z}(r+t)) =0$ for all $t \in \mathbb{Z}$ and by Theorem \ref{BottQ}, $H^{n-1}(Q_{n},\Omega_{Q_{n}}^{1}(t)) = 0,$ for $t \neq -n+2.$ Since, $h^{n-1}(Q_{n}, N_{\sF}^{\vee}(-n+2)) = 0,$ we conclude that $H^{n-1}_{\ast}(Q_{n},N_{\sF}^{\vee}) = 0.$ Therefore, $N_{\sF}^{\vee}$ is aCM.
						
					\end{proof}

				\end{document}